\documentclass{article}
\usepackage{latexsym, graphicx}
\usepackage{amsmath,amsthm, amssymb}
\pdfoutput=1

\newtheorem{theorem}{Theorem}[section]
\newtheorem{proposition}[theorem]{Proposition}
\newtheorem{lemma}[theorem]{Lemma}
\newtheorem{corollary}[theorem]{Corollary}

\newtheorem{observation}[theorem]{Observation}
\newtheorem*{lemmaagain}{Lemma~\ref{lem-sp3t} (repeated)}

\theoremstyle{definition}    
\newtheorem{definition}[theorem]{Definition}

\newcommand{\Sp}{\text{ST}}%plane semi-partial 3-trees, Appendix A
\newcommand{\Tr}{\text{Tr}}%tripods  
\newcommand{\B}{{\cal{B}}}%bubbles
\newcommand{\Em}{{\cal{E}}}%embedding
\newcommand{\cO}{{\cal O}}%velke O v asymptoticke notaci          
\newcommand{\wB}{{\widehat B}}%dual B   
\newcommand{\wf}{{\widehat \Phi}}%dual node of f 
\newcommand{\uhel}[2]{{\widehat{r_#1r_#2}}}%uhel sevreny dvema poloprimkama, viz Lemma o kresleni tripodu

\begin{document}

\title{The Planar Slope Number of Planar Partial 3-Trees of Bounded Degree} 

\author{V\'{\i}t Jel\'{\i}nek$^{1}$, 
        Eva Jel\'{\i}nkov\'a$^{1}$,  
        Jan Kratochv\'{\i}l$^{1,\,2}$,\\ 
        Bernard Lidick\'{y}$^{1}$, 
        Marek Tesa\v{r}$^{1}$,
        Tom\'{a}\v{s} Vysko\v{c}il$^{1,\,2}$\\[10pt]
\texttt{\char123
jelinek,eva,honza,bernard,tesulo,whisky\char125@kam.mff.cuni.cz}\\
\\
$^1$Department of Applied Mathematics\thanks{Supported by project
1M0021620838 of the Czech Ministry of Education and by SVV-2010-261313.}{}\ \ and\\
$^2$Institute for Theoretical Computer Science\thanks{Supported by grant
1M0545 of the Czech Ministry of Education.}\\ 
Charles University, Malostransk\'e n\'am\v est\'\i\ 25\\ Prague, Czech Republic
} 

\date{\today}

\maketitle

%%%%%%%%%%%%%%%%%%%%%%%%%%%%%%%%%%%%%%%%%%%%%%%%%%%%%%%%%%%%%%%%%%%%%5

\begin{abstract}
It is known that every planar graph has a planar embedding where edges are 
represented by non-crossing straight-line segments. We study the planar slope number, i.e., 
the minimum number of distinct edge-slopes in such a drawing of a planar graph 
with maximum degree~$\Delta$. 
We show that the planar slope 
number of every planar partial 3-tree and also every plane partial 3-tree is 
at most $\cO(\Delta^5)$. In particular, we answer the question of Dujmovi\' c 
et~al. [Computational Geometry 38 (3), pp. 194--212 (2007)]
whether there is a function $f$ such that plane maximal outerplanar 
graphs can be drawn using at most $f(\Delta)$ slopes.
\end{abstract}

%
%
%  ------------------  INTRO ----------------------------
% 
%

\section{Introduction}

The \emph{slope number} of a graph $G$ was introduced by Wade and Chu~\cite{wc1994}.
It is defined as the minimum number of distinct edge-slopes in a straight-line drawing of $G$.
Clearly, the slope number of $G$ is at most the number of edges of $G$, and it is at least
half of the maximum degree $\Delta$ of $G$.

Dujmovi\' c et~al.~\cite{dsw2004} asked whether there was a function $f$ such that
each graph with maximum degree $\Delta$ could be drawn using at most $f(\Delta)$ slopes.
In general, the answer is \emph{no} due to a result of Bar\' at et al.~\cite{bmw2006}.
Later, Pach and P\' alv\" olgyi~\cite{pp2006} and Dujmovi\' c et al.~\cite{dsw2007} proved that
for every $\Delta \geq 5$, there are graphs of maximum degree $\Delta$ that need an arbitrarily large number of slopes.

On the other hand, Keszegh et al.~\cite{kppt2007} proved that every graph of 
maximum degree three can be drawn using at most five slopes, and if we 
additionally assume that the graph is connected and has at least one vertex 
of degree less than three then four slopes suffice. Mukkamala and 
Szegedy~\cite{ms2009} have shown that four slopes also suffice for every 
connected cubic graph. Dujmovi\' c et al.~\cite{dsw2007} give a number of 
bounds in terms of the maximum degree: for interval graphs, cocomparability 
graphs, or AT-free graphs. All the results mentioned so far are related to 
straight-line drawings which are not necessarily non-crossing.

It is known that every planar graph $G$ can be drawn so that edges of $G$ are 
represented by non-crossing segments~\cite{f1948}. We call such a planar 
drawing a \emph{straight-line embedding} of~$G$. In this paper, we 
examine the minimum number of slopes in a straight-line embedding of a planar graph. 

In this paper, we make the (standard) distinction between \emph{planar 
graphs}, which are graphs that admit a plane embedding, and \emph{plane 
graphs}, which are graphs accompanied with a fixed prescribed combinatorial 
embedding, i.e., a prescribed face structure, including a prescribed outer 
face. Accordingly, we distinguish between the \emph{planar slope number} of a 
planar graph $G$, which is the smallest number of slopes needed to construct 
any straight-line embedding of $G$, as opposed to the \emph{plane slope 
number} of a plane graph~$G$, which is the smallest number of slopes needed 
to realize the prescribed combinatorial embedding of $G$ as a straight-line 
embedding.

The research of slope parameters related to plane embedding was initiated by 
Dujmovi\' c et~al.~\cite{dsw2004}. In~\cite{dsw-seg2007}, there are numerous 
results for the plane slope number of various classes of graphs. For instance, 
it is proved that every plane $3$-tree can be drawn using at most $2n$ slopes, 
where $n$ is its number of vertices. It is also shown that every $3$-connected 
plane cubic graph can be drawn using three slopes, except for the three edges 
on the outer face. 

Recently, Keszegh, Pach and  P\' alv\" olgyi~\cite{kpp2010}
have shown that any planar graph of maximum degree $\Delta$ can be drawn with
$2^{\cO(\Delta)}$ slopes. Their argument is based on a representation of planar
graph by touching disks.

In this paper, we study both the plane slope number and the planar slope 
number. The lower bounds of~\cite{bmw2006,dsw2007,pp2006} for bounded-degree 
graphs do not apply to our case, because the constructed graphs with large 
slope numbers are not planar. Moreover, the upper bounds 
of~\cite{kppt2007,ms2009} give drawings that contain crossings even for planar 
graphs.

For a fixed $k\in\mathbb N$, a \emph{$k$-tree} is defined recursively as follows.
A complete graph on $k$ vertices is a $k$-tree.
If $G$ is a $k$-tree and $K$ is a $k$-clique of $G$, then the graph formed by adding
a new vertex to $G$ and making it adjacent to all vertices of $K$ is also a $k$-tree.
A subgraph of a $k$-tree is called a \emph{partial $k$-tree}. 

We present several upper bounds on the plane and planar slope number in terms 
of the maximum degree $\Delta$. The most general result of this paper is the 
following theorem, which deals with plane partial 3-trees. 

\begin{theorem}\label{thm:tree}
The plane slope number of any plane partial 3-tree with maximum degree $\Delta$ 
is at most $\cO(\Delta^5)$.
\end{theorem}                                                                                             

Note that the above theorem implies that the planar slope number of any planar partial 
3-tree is also at most $\cO(\Delta^5)$. 

Since every outerplanar graph is also a partial 3-tree, the result above 
answers a question of Dujmovi\' c et~al.~\cite{dsw-seg2007}, who asked whether 
a plane maximal outerplanar graph can be drawn using at most $f(\Delta)$ 
slopes. 

Unlike the results of Keszegh, Pach and P\'alv\"olgyi~\cite{kpp2010}, our
arguments are only applicable to a restricted class of planar graphs. On the
other hand, our bound is polynomial in $\Delta$ rather than exponential, and
moreover, our proof is constructive.

In the special case of series-parallel graphs of maximum degree at most 3, we 
are able to prove a better (in fact optimal) upper bound.

\begin{theorem}\label{thm:spmax3}
Any series-parallel graph with maximum degree at most 3 has planar slope number at most 3.
\end{theorem}                                                 

Extended abstract of this paper was presented at Graph Drawing 2009~\cite{jjktv2009}.
Note that in that version Theorem~\ref{thm:tree} was stated with
bound $\cO(2^{\cO(\Delta)})$.

%
%
%  ------------------  PRELIMINARIES ----------------------------
% 
%

\section{Preliminaries}

Let us introduce some basic terminology and notation that will be used 
throughout this paper.

Let $s$ be a segment in the plane. The smallest angle $\alpha \in [0,\pi)$ such
that $s$ can be made horizontal by a clockwise rotation by $\alpha$, is called
the \emph{slope of $s$}. The \emph{directed slope} of a directed segment is an
angle $\alpha' \in [0,2\pi)$ defined analogously.

A plane graph is called a \emph{near triangulation} if all its faces, except 
possibly the outer face, are triangles.

%
%
%  ------------------  PLANE PARTIAL 3-TREES ----------------------------
% 
%
\section{Plane partial 3-trees}

In this section we present the proof of Theorem~\ref{thm:tree}. We start with
some observations about the structure of plane 3-trees.
Throughout this section, we assume that $\Delta$ is a fixed integer.

It is known that any plane 3-tree can be generated from a triangle by a 
sequence of vertex-insertions into inner faces. Here, a 
\emph{vertex-insertion} is an operation that consists of creating a new vertex 
in the interior of a face, and then connecting the new vertex to all the three 
vertices of the surrounding face, thus subdividing the face into three new 
faces.

For a plane partial 3-tree $G$ we define the \emph{level} of a vertex $v$ as 
the smallest integer $k$ such there is a set $V_0$ of $k$ vertices of $G$ with 
the property that $v$ is on the outer face of the plane graph $G-V_0$. 
Let $G$ be a plane partial 3-tree. An edge of $G$ is called \emph{balanced} 
if it connects two vertices of the same level of~$G$. An edge that is not 
balanced is called \emph{tilted}. Similarly, a face of $G$ whose all vertices 
belong to the same level is called balanced, and any other face is called tilted.
In a plane 3-tree, the level of a vertex $v$ can also be equivalently defined as the 
length of the shortest path from $v$ to a vertex on the outer face. However, 
this definition cannot be used for plane partial 3-trees.

Note that whenever we insert a new vertex $v$ into an inner face of a 3-tree, 
the level of $v$ is one higher than the minimum level of its three neighbors; 
note also that the level of all the remaining vertices of the 3-tree is not 
affected by the insertion of a new vertex.

Let $u,v$ be a pair of vertices forming an edge. A \emph{bubble} over $uv$ is 
an outerplanar plane near triangulation that contains the edge $uv$ on the 
boundary of the outer face. The edge $uv$ is called the \emph{root} of the 
bubble. A \emph{trivial bubble} is a bubble that has no other edge apart from 
the root edge. A \emph{double bubble} over $uv$ is a union of two bubbles over 
$uv$ which have only $u$ and $v$ in common and are attached to $uv$ from its 
opposite sides. A \emph{leg} is a graph $L$ created from a path $P$ by adding a 
double bubble over every edge of $P$. The path $P$ is called the \emph{spine 
of $L$} and the endpoints of $P$ are also referred to as the endpoints of the 
leg. Note that a single vertex is also considered to form a leg. 

A \emph{tripod} is a union of three legs which share a common endpoint. A 
\emph{spine} of a tripod is the union of the spines of its legs. Observe that 
a tripod is an outerplanar graph. The vertex that is shared by 
all the three legs of a tripod is called \emph{the central vertex}.

Let $G$ be a near triangulation, let $\Phi$ be an inner face of $G$. Let $T$ 
be a tripod with three legs $X,Y,Z$ and a central vertex $c$. An 
\emph{insertion of tripod $T$ into the face $\Phi$} is the operation performed 
as follows. First, insert the central vertex $c$ into the interior of $\Phi$ 
an connect it by edges to the three vertices of~$\Phi$. This subdivides $\Phi$ 
into three subfaces. Extend $c$ into an embedding of the whole tripod $T$, by 
embedding a single leg of the tripod into the interior of each of the three 
subfaces. Next, connect every non-central vertex of the spine of the tripod to 
the two vertices of $\Phi$ that share a face with the corresponding leg. 
Finally, connect each non-spine vertex $v$ of the tripod to the single vertex 
of $\Phi$ that shares a face with~$v$. See Figure~\ref{fig:tripod}. Observe 
that the graph obtained by a tripod insertion into $\Phi$ is again a near 
triangulation.

\begin{figure}
\begin{centering}
 \includegraphics{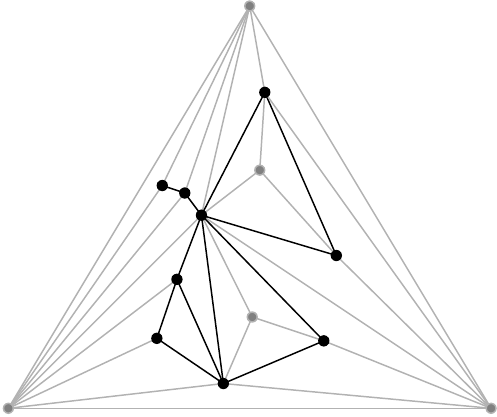}
 \caption{An example of a tripod consisting of vertices of level 1 in a plane 
3-tree.\label{fig:tripod}}
\end{centering}
\end{figure}

\begin{lemma}\label{lem:tripodOK}
Let $G$ be a graph. The following statements are equivalent:
\begin{enumerate}
  \item 
$G$ is a plane 3-tree, i.e., $G$ can be created from a triangle by a sequence 
of vertex insertions into inner faces.
  \item 
$G$ can be created from a triangle by a sequence of tripod insertions into inner faces.
  \item 
$G$ can be created from a triangle by a sequence of tripod insertions into 
balanced inner faces.
\end{enumerate}
\end{lemma}
\begin{proof}
Clearly, (3) implies (2).

To observe that (2) implies (1), it suffices to notice that a tripod insertion 
into a face $\Phi$ can be simulated by a sequence of vertex insertions: first 
insert the central vertex of a tripod into $\Phi$, then insert the vertices of 
the spine into the resulting subfaces, and then create each bubble by 
inserting vertices into the face that contains the root of the bubble and its 
subsequent subfaces.

To show that (1) implies (3), proceed by induction on the number of levels 
in~$G$. If $G$ only has vertices of level $0$, then it consists of a single 
triangle and there is nothing to prove. Assume now that the $G$ is a graph 
that contains vertices of $k>0$ distinct levels, and assume that any 3-tree 
with fewer levels can be generated by a sequence of balanced tripod insertions 
by induction. 

We will show that the vertices of level $k$ induce in $G$ a subgraph 
whose every connected component is a tripod, and that each of these tripod is 
inserted inside a triangle whose vertices have level $k-1$. 

Let $C$ be a connected component of the subgraph induced in $G$ by the 
vertices of level $k$. Let $v_1,v_2,\dotsc,v_m$ be the vertices of $C$, in the 
order in which they were inserted when $G$ was created by a sequence of vertex 
insertions. Let $\Phi$ be the triangle into which the vertex $v_1$ was 
inserted, and let $x,y$ and $z$ be the vertices of $\Phi$. Necessarily, all 
three of these vertices have level $k-1$. Each of the vertices 
$v_2,\dotsc,v_m$ must have been inserted into the interior of $\Phi$, and each 
of them must have been inserted into a face that contained at least one of the 
three vertices of $\Phi$. 

Note that at each point after the insertion of $v_1$, there are exactly three 
faces inside $\Phi$ that contain a pair of vertices of $\Phi$; each of these 
three faces is incident to an edge of $\Phi$. Whenever a vertex $v_i$ is 
inserted into such a face, the subgraph induced by vertices of level $k$ grows 
by a single edge. These edges form a union of three paths that share the vertex 
$v_1$ as their common endpoint. 

On the other hand, when a vertex $v_i$ is inserted into a face formed by a single 
vertex of $\Phi$ and a pair of previously inserted vertices $v_j$, $v_\ell$, 
then the graph induced by vertices of level $k$ grows by two new edges $v_iv_j$ and $v_iv_\ell$, as well as a new triangular face with vertices $v_i,v_j,v_\ell$. 

With these observations, it is easily checked (e.g., by induction on $i$) that 
for every $i\ge1$, the subgraph of $G$ induced by the vertices 
$v_1,\dotsc,v_i$ is a tripod inserted into $\Phi$. From this fact, it follows 
that the whole graph $G$ may have been created by a sequence of tripod 
insertions into balanced faces.
\end{proof}                    

Note that when we insert a tripod into a balanced face, all the vertices of 
the tripod will have the same level (which will be one higher than the level 
of the face into which we insert the tripod). In particular, each balanced 
face we create by this insertion is an inner face of the inserted tripod.

Recall that a plane partial 3-tree is a plane graph that is a subgraph of a 
3-tree. 
Kratochv\'\i l and Vaner~\cite{kv2010} 
have shown that every plane 
partial 3-tree $G$ is in fact a subgraph of a plane 3-tree. Furthermore, if a 
plane partial 3-tree $G$ has at least three vertices, it is in fact a spanning 
subgraph of a plane 3-tree, i.e., it can be extended into a plane 3-tree by 
only adding edges. 

Unfortunately, the plane 3-tree that contains a plane partial 3-tree $G$ may 
in general require arbitrarily large vertex-degrees, even if the maximum 
degree of $G$ is bounded. Thus, the result of Kratochv\'\i l and Vaner does 
not allow us to directly simplify the problem to plane 3-trees drawing.

To overcome this difficulty, we introduce the notion of `plane semi-partial 
3-tree', which can be seen as an intermediate concept between plane 3-trees 
and plane partial 3-trees.

\begin{definition}
A graph $G$ is called \emph{a plane semi-partial 3-tree} if 
$G$ is obtained from a plane 3-tree $H$ by erasing some of the tilted edges of~$H$.
\end{definition}

Our goal is to prove that every plane partial 3-tree of maximum degree 
$\Delta$ can be drawn with at most $\cO(\Delta^5)$ slopes. We obtain this 
result as a direct consequence of two main propositions, stated below.

\begin{proposition}\label{pro-semi}
Any connected plane partial 3-tree of maximum degree $\Delta$ is a subgraph of a 
plane semi-partial 3-tree of maximum degree at most $37\Delta$.
\end{proposition}   

\begin{proposition}
\label{pro-sp3t} For every $\Delta$ there is a set $S$ of at most 
$\cO(\Delta^5)$ slopes with the property that any plane semi-partial 3-tree 
of maximum degree $\Delta$ has a straight-line embedding whose edge-slopes all belong to~$S$.
\end{proposition}

We begin by proving Proposition~\ref{pro-semi}.

\subsection{Proof of Proposition~\ref{pro-semi}}

We begin by a simple lemma, which shows that the deletion of tilted edges 
from a plane 3-tree does not affect the level of vertices.

\begin{lemma}\label{lem-level}
Let $H=(V,E)$ be a plane 3-tree, let $T$ be a set of tilted edges of $H$, let 
$G=(V,E\setminus T)$ be a semi-partial 3-tree. Let $v$ be a vertex of level $k$ 
with respect to $H$. Then $v$ has level $k$ in $G$ as well.
\end{lemma}                                                
\begin{proof}                                                               
Fix a vertex $v$ of level $k$ in $H$. Of course, the deletion of an edge 
may only decrease the level of a vertex, so $v$ has level at most $k$ in $G$. 
On the other hand, it follows from Lemma~\ref{lem:tripodOK}
that every vertex of level $k$ in $H$ is separated from the outer face by $k$ 
nested triangles $C_0, C_1,\dotsc C_{k-1}$, where $C_i$ is a triangle formed 
by balanced edges that belong to level $i$. Since every balanced edge of $H$ 
belongs to $G$ as well, we know that all the triangles $C_0, C_1,\dotsc 
C_{k-1}$ belong to $G$, showing that $v$ has level at least $k$. It follows 
that the level of $v$ is preserved by the deletion of tilted edges.
\end{proof}                                                        

Let $G=(V,E)$ be a plane semi-partial 3-tree obtained from a plane 3-tree 
$H=(V,E')$ by the deletion of several tilted edges. As a consequence of the 
previous lemma, we see that an edge $e\in E$ is tilted in $G$ if and only if 
it is tilted in $H$.

Assume now that $F$ is a connected plane partial 3-tree with maximum 
degree~$\Delta\ge 1$ and at least three vertices. Our goal is to show that 
there is a plane semi-partial 3-tree $G$ with maximum degree at most 
$37\Delta$ that contains $F$ as a spanning subgraph. The following definition 
introduces the key notion of our proof.

\begin{definition}
Let $F$ be a connected plane partial 3-tree with maximum degree $\Delta$, and 
let $k$ be an integer. We say that a 3-tree $H$ \emph{correctly covers $F$ up 
to level $k$}, if the following conditions are satisfied:
\begin{itemize}
\item $F$ is a spanning subgraph of $H$.
\item Let $V^{\le k}$ denote the set of vertices that have level at most $k$ in 
$H$. For every vertex $v\in V^{\le k}$ there are at most $36\Delta$ balanced 
edges of $H$ that are incident to $v$.
\end{itemize}                        
Furthermore, we say that $H$ \emph{correctly covers $F$ at all levels} if, 
for any $k$, $H$ correctly covers $F$ up to level~$k$.
\end{definition}
                                   
As mentioned before, 
Kratochv\'\i l and Vaner~\cite{kv2010}
have shown that 
every plane partial 3-tree $F$ is a spanning subgraph of a plane 3-tree $H$. 
Note that such a 3-tree $H$ correctly covers $F$ up to level 0, because every 
vertex at level 0 is adjacent to two balanced edges.

Our proof of Proposition~\ref{pro-semi} is based on the following lemma.

\begin{lemma}\label{lem-cover}
For every connected partial 3-tree $F$ there is a 3-tree $H$ that correctly 
covers $F$ at all levels.
\end{lemma}
                         
Before we prove the lemma, let us show how it implies 
Proposition~\ref{pro-semi}.

\begin{proof}[Proof of Proposition~\ref{pro-semi} from Lemma~\ref{lem-cover}]
Let $F$ be a plane partial 3-tree of maximum degree $\Delta$, and let $H$ be 
the 3-tree that correctly covers $F$ at all levels. Define a semi-partial 
3-tree $G$ which is obtained from $H$ by erasing all the tilted edges of $H$ 
that do not belong to $F$. By construction, $G$ is a semi-partial 3-tree that 
contains $F$ as a subgraph. Moreover, every vertex of $G$ is adjacent to at 
most $\Delta$ tilted edges and at most $36\Delta$ balanced edges, so $G$ has 
maximum degree at most $37\Delta$.
\end{proof}    

Let us now turn to the proof of Lemma~\ref{lem-cover}.

\begin{proof}   
Let $F$ be a partial 3-tree with maximum degree $\Delta$, and assume for 
contradiction that there is no graph $H$ that would correctly cover~$F$. Let 
$k$ be the largest integer such that there is a graph $H$ that correctly 
covers $F$ up to level~$k$. We have seen that $k\ge 0$. On the other hand, we 
clearly have $k<|V(F)|$. Thus, $k$ is well defined.

Fix a graph $H$ correctly covering $F$ up to level $k$. By our assumption, 
$H$ has vertices of level greater than~$k$.  We will now define a 3-tree $H'$ 
that correctly covers $F$ up to level $k+1$, which contradicts the maximality 
of~$k$. 

Note that it is sufficient to ensure that 
$H'$ is constructed by a sequence of balanced tripod insertions in 
which all the tripods inserted at level at most $k+1$ have degrees 
bounded by $36\Delta$.

We construct $H'$ in such a way that it coincides with~$H$ on vertices of 
level at most $k$; more precisely, if $u$ and $v$ are two vertices of level at 
most $k$ in $H$, then $u$ and $v$ are connected by an edge of $H'$ if and only 
if they are connected by an edge of $H$. Notice that this property guarantees 
that the vertices at level at most $k$ in $H$ are at the same level in $H'$ as 
in $H$. Let $H^{\le k}$ be the subgraph of $H$ induced by the vertices of 
level at most $k$. $H^{\le k}$ is a 3-tree.

Let $\Phi$ be a balanced face of $H^{\le k}$ formed by vertices at level $k$ 
which contains at least one vertex of $H$ at level $k+1$ in its interior. Note 
that at least one such face exists, since we assumed that at least one vertex 
has level greater than $k$ in $H$. For any such face $\Phi$, we will modify 
the sequence of tripod-insertions performed inside $\Phi$, such that the 
tripod inserted into this face has maximum degree at most $36\Delta$, while 
the modified graph will still contain $F$ as a subgraph. By doing this modification 
inside every nonempty balanced face at level $k$, we will eventually obtain a 
graph $H'$ that correctly covers $F$ up to level $k+1$.

Fix $\Phi$ to be a balanced face at level $k$ with nonempty interior. Let 
$T\subset H$ be the tripod that has been inserted into $\Phi$ during the 
construction of $H$. Let $V_T$ and $E_T$ be the vertices and the edges of $T$. 
We will now define a modified tripod $T'$ on the vertex set $V_T$, 
satisfying the required degree bound. We will then show that the sequence of 
tripod insertions that have been performed inside $T$ during the 
construction of $H$ can be transformed into a sequence of tripod insertions 
inside $T'$, where the new sequence of insertions yields a graph $H'$ that 
contains $F$ as a subgraph.                              

We define $T'$ by the following rules.
\begin{enumerate}         
\item All the edges of $T$ that belong to $F$ are also in~$T'$.
\item All the edges of $T$ that belong to the boundary of the outer face 
of $T$ also belong to~$T'$. These edges form the boundary of the outer face 
of~$T'$.
\item All the edges that form the spine of $T$ also belong to $T'$ and they 
form its spine.
\item 
Let $\Psi$ be an internal face of the tripod $T$. Let $u$, $v$ and $w$ be the 
three vertices of~$\Psi$. Assume that both $u$ and $v$ are connected by an 
edge of $F$ to a vertex in the interior of $\Psi$ (not necessarily both of 
them to the same vertex). In such case, add the edge $uv$ to $T'$.
\item 
Let $T'_0$ be the graph formed by all the edges added to $T'$ by the previous 
four rules. Note that $T'_0$ is an outerplanar graph with the same outer face 
as $T$. However, not all the inner faces of $T'_0$ are necessarily triangles, 
so $T'_0$ is not necessarily a tripod. Assume that $T'_0$ has an inner face 
with more than three vertices, and that $v_0, v_1,\dotsc, v_r$ are the 
vertices of this face, listed in cyclic order. We form the path $v_1,v_r, v_2, 
v_{r-1}, v_3, v_{r-2},\dotsc$ whose edges triangulate the face of $T'_0$. We 
add all the edges of this path into $T'$. We do this for every internal face 
of $T'_0$ that has more than three vertices. The resulting graph $T'$ is 
clearly a tripod.
\end{enumerate}

Let us now argue that the tripod $T'$ has maximum degree at most 
$36\Delta$. Let $v\in V_T$ be any vertex of this tripod. Let us estimate 
$\deg_{T'}(v)$, by counting the edges adjacent to $v$ that were added to $T'$ 
by the rules above. Clearly, there are at most $\Delta$ such edges that were 
added by the first rule, and at most nine such edges that were added by the 
second and third rule.

We claim that there are at most $2\Delta$ edges incident with $v$ added by the 
fourth rule. To see this, notice that if $e=uv$ is an edge added by this rule, 
then at least one of the two faces of $T$ that are incident to $e$ must 
contain in its interior an edge $e'$ of $F$ that is incident to $v$. In such situation,
we say that $e'$ is \emph{responsible} for the insertion of $e$ into~$T'$.
Clearly, an edge of $F$ may be responsible for the insertion of at most two edges incident with~$v$. Since $v$ has degree at most $\Delta$ in $F$, this shows that at most
$2\Delta$ edges incident with $v$ are added to $T_0'$ by the fourth rule. 
Consequently, $T'_0$ has maximum degree at most $3\Delta+9$.

To estimate the number of edges added to $T'$ by the fifth rule, it is 
sufficient to observe that in every internal face of $T'_0$ whose boundary 
contains $v$ there are at most two edges of $T'$ incident to $v$ added by the 
fifth rule. Thus, $\Delta(T')\le 3\Delta(T'_0)\le 9\Delta+27\le 36\Delta$, as claimed. 
                  
Having thus defined the tripod $T'$, we modify the graph $H$ as follows. We 
remove all the vertices appearing in the interior of the face $\Phi$ of 
$H^{\le k}$; that is, we remove the tripod $T$ as well as all the vertices 
inserted inside $T$. Instead, as a first step towards the construction of 
$H'$, we insert $T'$ inside $\Phi$.

To finish the construction of $H'$, we need to insert the vertices of level 
greater than $k+1$ into the faces of $T'$, so that the resulting graph 
contains $F$ as a subgraph. We perform this insertion separately inside every 
face of $T'_0$. Note that $T'_0$ is a subgraph of $T$ as well as a subgraph of 
$T'$, and that each internal face of $T'_0$ is a union of several faces of $T'$. Let $\Psi$ be a face of $T'_0$. If $\Psi$ is a triangle, then $\Psi$ is 
in fact a face of $T'$ as well as a face of~$T$. If $T$ contains a subgraph 
$H_\Psi$ inside $\Psi$, we define $H'$ to contain the same subgraph inside 
$\Psi$ as well. Since $H_\Psi$ has been created by a sequence of 
tripod insertions inside $\Psi$, we can perform the same sequence of tripod 
insertion again inside the same face during the construction of $H'$.

Assume now that $\Psi$ is not a triangle. In the graph $H$, the face $\Psi$ is 
subdivided into a collection of triangular faces 
$\Psi_1, \Psi_2, \dotsc,\Psi_k$. Let $H_i$ be the subgraph of $H$ appearing 
inside the face $\Psi_i$ in $H$. We know that each $H_i$ is a result of a 
sequence of tripod insertions. 

Let us use the following terminology: if there is an edge of $F$ that connects 
a vertex of $H_i$ to a vertex $v$ on the boundary of $\Psi$, we say that 
$H_i$ is \emph{adjacent} to~$v$. Since the graph $F$ is connected, each 
nonempty graph $H_i$ must be adjacent to at least one vertex on the boundary of 
$\Psi$. Observe that if $H_i$ is adjacent to two distinct vertices $u$ and $v$ 
on the boundary of $\Psi$, then the edge that connects $u$ and $v$ must belong 
to $T'_0$ by the fourth rule in the construction of $T'$. In particular, $u$ 
and $v$ appear consecutively on the boundary of $\Psi$. This also shows that 
$H_i$ cannot be adjacent to three distinct vertices of $\Psi$, since we 
assumed that $\Psi$ is not a triangle. 

Consider now the tripod $T'$. In this tripod, the face $\Psi$ is triangulated 
into a collection of faces $\Psi'_1, \Psi'_2,\dotsc,\Psi'_k$. Each of these 
triangular faces has at least one edge of $T'_0$ on its boundary. We will 
insert the graphs $H_1, H_2,\dotsc, H_k$ into these faces, by performing for 
each $H_i$ a sequence of tripod insertions which generates $H_i$ inside one of 
the faces $\Psi'_1, \Psi'_2,\dotsc,\Psi'_k$. 

To ensure that the resulting graph will contain $F$ as a subgraph, it suffices 
to guarantee that whenever $H_i$ is adjacent to a vertex $v\in\Psi$, it will 
be inserted into a face $\Psi'_j$ that contains $v$ on its boundary. Such a 
face always exists, since each $H_i$ is adjacent to at most two vertices of 
$\Psi$, and if it is adjacent to two vertices $u, v$, then the two vertices 
must be connected by an edge on the boundary of $\Psi$, which implies that 
there is a face $\Psi'_j$ that contains both $u$ and $v$ on its boundary.

It may happen that two distinct graphs $H_i$ and $H_j$ need to be inserted 
into the same face $\Psi'_\ell$. In such case, the first graph is inserted 
directly into $\Psi'_\ell$, thus partitioning it into several smaller 
triangular subfaces, while all subsequent graphs that need to be inserted into 
$\Psi'_\ell$ are inserted into an appropriately chosen subface of 
$\Psi'_\ell$. This subface need not be balanced. We choose this subface in 
such a way that we preserve the cyclic order of edges of $F$ around every 
vertex $v$ on the boundary of $\Psi$. 

After we perform the construction above inside every face $\Psi$ of $T'_0$, we 
obtain a plane 3-tree $H'$ that correctly covers $F$ up to level $k+1$. This 
completes the proof of the lemma.
\end{proof}

%%%%%%%%%%%%%%%%%%%%%%%%%%%%%%%%%%%%%%%%%%%%%%%%%%%%%%%%%%%%%%%%%%%%%%%%%%%%%%%%
%%%%%%%%%%%%%%%%%%%%%%%%%%%%%%%%%%%%%%%%%%%%%%%%%%%%%%%%%%%%%%%%%%%%%%%%%%%%%%%%
%%%%%%%%%%%%%%%%%%%%%%%%%%%%%%%%%%%%%%%%%%%%%%%%%%%%%%%%%%%%%%%%%%%%%%%%%%%%%%%%
%%%%%%%%%%%%%%%%%%%%%%%%%%%%%%%%%%%%%%%%%%%%%%%%%%%%%%%%%%%%%%%%%%%%%%%%%%%%%%%%
%%%%%%%%%%%%%%%%%%%%%%%%%%%%%%%%%%%%%%%%%%%%%%%%%%%%%%%%%%%%%%%%%%%%%%%%%%%%%%%%
%%%%%%%%%%%%%%%%%%%%%%%%%%%%%%%%%%%%%%%%%%%%%%%%%%%%%%%%%%%%%%%%%%%%%%%%%%%%%%%%
%%%%%%%%%%%%%%%%%%%            Proposition 2                 %%%%%%%%%%%%%%%%%%%
%%%%%%%%%%%%%%%%%%%%%%%%%%%%%%%%%%%%%%%%%%%%%%%%%%%%%%%%%%%%%%%%%%%%%%%%%%%%%%%%
%%%%%%%%%%%%%%%%%%%%%%%%%%%%%%%%%%%%%%%%%%%%%%%%%%%%%%%%%%%%%%%%%%%%%%%%%%%%%%%%
%%%%%%%%%%%%%%%%%%%%%%%%%%%%%%%%%%%%%%%%%%%%%%%%%%%%%%%%%%%%%%%%%%%%%%%%%%%%%%%%

\subsection{Proof of Proposition~\ref{pro-sp3t}}

To complete the proof of our main result, it remains to show that every plane 
semi-partial 3-tree of bounded maximum degree has a straight-line embedding with a bounded 
number of slopes. 

We start with a brief overview of the construction. We will use the fact that 
a plane semi-partial 3-tree $G$ can be decomposed into tripods formed by 
vertices of the same level, with each tripod $T$ of level $k\ge 1$ being inserted into a 
triangle $\Phi$ formed by vertices of level~$k-1$. The triangle $\Phi$ is 
itself an inner face of a tripod of level~$k-1$.

The tripods appearing in this decomposition of $G$ may be arbitrarily large. 
However, a tripod $T$ of level $k\ge 1$ has only a bounded number of vertices 
that are adjacent to a vertex of the triangle $\Phi$ of level~$k-1$. These 
vertices of $T$ will be called \emph{relevant vertices}.

Given a tripod $T$ in the decomposition of $G$, we will construct an 
embedding of $T$ that only uses edge-slopes from a set of slopes $S'$ and 
moreover, all the relevant vertices of $T$ are embedded on points from a set 
of points $P'$, where the sets $S'$ and $P'$ are independent of $T$ and 
their size is polynomial in~$\Delta$. 

We will then show that these embeddings of tripods (after a suitable scaling) 
can be nested into each other to provide the embedding of the whole graph~$G$. 
We will argue that the number of edge-slopes in this embedding of $G$ is 
bounded. This will follow from the fact that the balanced edges of $G$ belong 
to a tripod and their slope belongs to~$S'$, while the slopes of the tilted 
edges only depend on the positions of the relevant vertices of a tripod $T$ 
and on the shape of the triangle~$\Phi$ surrounding~$T$. Since the relevant 
vertices can only have a bounded number of positions, and the triangle $\Phi$ 
is formed by balanced edges and hence may have only a bounded number of 
shapes, we will conclude that the tilted edges may only determine a bounded number 
of slopes.

Let us now describe the construction in detail. We recall that $\Delta$ is a 
fixed constant throughout this section, and we let $\Sp(\Delta)$ denote the 
set of plane semi-partial 3-trees of maximum degree at most~$\Delta$. Any 
graph $G\in\Sp(\Delta)$ can be created by a sequence of \emph{partial tripod 
insertions} into balanced faces, where a partial tripod insertion is defined 
in the same way as an ordinary tripod insertion, except that some of the 
tilted edges are omitted when the new tripod is inserted. 

Choose a graph $G\in \Sp(\Delta)$, and assume that $T$ is a tripod that is 
used in the construction of $G$ by a sequence of partial tripod insertions. 
Let $\{x,y,z\}$ be the triangle in $G$ into which the tripod $T$ has been 
inserted. We say that a vertex $v$ of $T$ is \emph{relevant} if $v$ is 
connected by an edge of $G$ to at least one of the vertices $x,y$ or $z$. 
Since each of the three vertices $x$, $y$ and $z$ has degree at most 
$\Delta$, the tripod $T$ has at most $3\Delta$ relevant vertices. Let us 
further say that a bubble of $T$ is relevant if it contains at least one 
relevant vertex. Since every vertex of $T$ is contained in at most six 
bubbles, we see that $T$ has at most $18\Delta$ relevant bubbles.

We will use the term \emph{labelled tripod of degree $\Delta$} to denote a 
tripod $T$ with maximum degree at most $\Delta$, together with an associated 
set of at most $3\Delta$ relevant vertices of $T$. Let $\Tr(\Delta)$ be the 
(infinite) set of all the labelled tripods of degree~$\Delta$. Similarly, a 
\emph{labelled bubble of degree $\Delta$} is a bubble of maximum degree at 
most $\Delta$, together with a prescribed set of at most $3\Delta$ relevant 
vertices. $\B(\Delta)$ denotes the set of all such labelled bubbles.

Let $\Em_T$ be an embedding of a tripod in the plane, and let $v$ be a vertex 
of $\Em_T$. Let $\alpha\in\langle0,2\pi)$ be a directed slope. We say that the 
vertex $v$ has \emph{visibility in direction $\alpha$}, if the ray starting in 
$v$ and having direction $\alpha$ does not intersect $\Em_T$ in any point 
except~$v$. 

Throughout the rest of this section, let $\varepsilon$ denote the 
value~$\pi/100$ (any sufficiently small integral fraction of $\pi$ is suitable 
here).

Our proof of Proposition~\ref{pro-sp3t} is based on the following key lemma.

\begin{lemma}[Tripod Drawing Lemma]\label{lem-sp3t}    
For every $\Delta$ there is a set of slopes $S$ of size $\cO(\Delta^3)$, a set 
of points $P$ of size $\cO(\Delta^2)$, and a set of triangles $R$ of size 
$\cO(\Delta^3)$, such that every labelled tripod $T\in\Tr(\Delta)$ has a 
straight-line embedding $\Em_T$ with the following properties:
\begin{enumerate}
\item 
The slope of any edge in the embedding $\Em_T$ belongs to~$S$.
\item 
Each relevant vertex of $\Em_T$ is embedded on a point from~$P$.
\item 
Each internal face of $\Em_T$ is homothetic to a triangle from~$R$.
\item                                                           
The central vertex of $\Em_T$ is embedded in the origin of the 
plane.
\item 
Any vertex of $\Em_T$ is embedded at a distance at most $1$ from the origin.
\item
Each spine of $T$ is embedded on a single ray starting from the origin. The three 
rays containing the spines have directed slopes $0$, $2\pi/3$ and~$4\pi/3$. 
Let these three rays be denoted by $r_1$, $r_2$ and $r_3$, respectively.  
\item 
Let $\uhel{i}{j}$ denote the closed convex region whose boundary is formed by 
the rays $r_i$ and~$r_j$. Any relevant vertex of $\Em_T$ embedded in the 
region $\uhel{1}{2}$ (or $\uhel{2}{3}$, or $\uhel{1}{3}$) has visibility in 
any direction from the set $\langle \varepsilon, 2\pi/3-\varepsilon\rangle$ 
(or $\langle 2\pi/3+\varepsilon, 4\pi/3-\varepsilon\rangle$, or $\langle 
4\pi/3+\varepsilon,2\pi-\varepsilon\rangle$, respectively). 

Note that the three regions $\uhel{1}{2}$, $\uhel{2}{3}$ and $\uhel{1}{3}$ 
are not disjoint. For instance, if a relevant vertex of $T$ is embedded 
on the ray $r_1$, it belongs to both $\uhel{1}{2}$ and $\uhel{1}{3}$, and 
hence it must have visibility in any direction from the set $\langle 
\varepsilon, 2\pi/3-\varepsilon\rangle\cup \langle 
4\pi/3+\varepsilon,2\pi-\varepsilon\rangle$. 
\end{enumerate}
\end{lemma}

Before we prove Lemma~\ref{lem-sp3t}, we show how the lemma implies 
Proposition~\ref{pro-sp3t}.

\begin{proof}[Proof of Proposition~\ref{pro-sp3t} from Lemma~\ref{lem-sp3t}]

Let $S$ be the set of slopes, $P$ be the set of points and $R$ be the set
of triangles from  Lemma~\ref{lem-sp3t}. 
Let $S'$ be the set of all the slopes that differ from a 
slope in $S$ by an integer multiple of $\varepsilon$. Note that 
$|S'|\le\frac{\pi}{\varepsilon}|S|$. Let $P'$ be the (finite) set of points 
that can be obtained by rotating a point in $P$ around the origin by an 
integral multiple of~$\varepsilon$. Let $R'$ be the (finite) set of triangles
that is obtained by rotating the triangles in $R$ by an integral multiple of~$\varepsilon$.

We will show that any graph $G\in\Sp(\Delta)$ has a straight-line embedding 
where the slopes of balanced edges belong to $S'$ and the slopes of tilted 
edges also belong to a finite set which is independent of~$G$.

Let $T$ be a labelled tripod used in the construction of the graph $G$. Assume 
that $T$ is inserted into a triangle formed by three vertices $x, y, z$ (see 
Figure~\ref{fig-pointp}). Let $\tau$ be the triangle formed by the three points 
$x,y,z$. Assume that the three vertices are embedded in the plane. Without 
loss of generality, assume that the triangle $\tau$ has acute angles by the 
vertices $y$ and $z$, and the three vertices $xyz$ appear in counterclockwise 
order around the boundary of~$\tau$. Thus the altitude of $\tau$ from the 
vertex $x$ intersects the segment $yz$ on a point $p$ which is in the interior 
of the segment $yz$. Let $\eta$ be the slope of the (directed) segment $yz$.   
\begin{figure}
\hfil\includegraphics{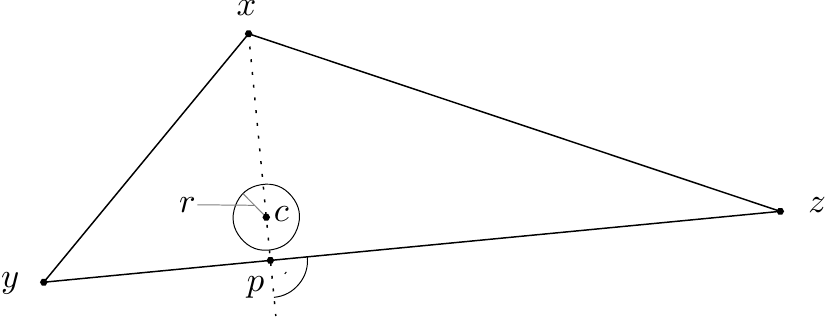}
\caption{Illustration of the proof of Proposition~\ref{pro-sp3t}}\label{fig-pointp}
\end{figure}

We can find a point $c$ in the interior of the triangle $\tau$, and a positive 
real number $r=r(\tau)$, such that for any point $v$ at a distance at most $r$ from 
$c$, the following holds:
\begin{enumerate}
\item $v$ is in the interior of $\tau$
\item the slope of the segment $vx$ differs from the slope of the segment $px$ 
(which is equal to $\eta+\pi/2$) by less than $\varepsilon$
\item 
the slope of the segment $vy$ differs from the slope of the segment $py$ 
(which is  equal to $-\eta$) by less than $\varepsilon$
\item                                   
the slope of the segment $vz$ differs from the slope of the segment $pz$ 
(which is  equal to $\eta$) by less than $\varepsilon$
\end{enumerate}
Indeed, it suffices to choose $c$ sufficiently close to the point $p$ and set 
$r$ sufficiently small, and all the above conditions will be satisfied.

Consider now the embedding $\Em_T$ of $T$. Place the center of the tripod on 
the point $c$, and scale the whole embedding by the factor $r$, so that it 
fits inside the triangle $\tau$. In view of the four conditions above, and in 
view of the seventh part of Lemma~\ref{lem-sp3t}, it is not difficult to 
observe that we may rotate the (scaled) embedding of $T$ around the point $c$ 
by an integral multiple of $\varepsilon$ in such a way that every relevant 
vertex $v\in T$ has visibility towards all its neighbors among the three 
vertices $x,y, z$. Thus, we are able to embed all the necessary tilted edges 
of $G$ between $xyz$ and $T$ as straight line segments. 

Note that in our embedding, all the balanced edges of $T$ have slopes from the 
set $S'$, and all its internal faces are homothetic to the triangles from the set $R'$. Furthermore, any tilted edge has one endpoint in the set $\{x,y,z\}$ 
and another endpoint in the set $c+rP'$ (the set $P'$ scaled $r$-fold and 
translated in such a way that the origin is moved to $c$). Hence
any labelled tripod  $T\in\Tr(\Delta)$ can be inserted inside the 
triangle $xyz$ in such a way that the slopes of the edges always belong to the 
same finite set which depends on the triangle $xyz$ but not on the tripod $T$.
Note that the triangle $xyz$ may be arbitrarily thin, in particular it can have inner
angles smaller than $\varepsilon$.

Let us now show how the above construction yields an embedding of the whole 
graph $G$. For every such 
triangle $\tau\in R'$, fix the point $c=c(\tau)$ and the radius $r=r(\tau)$ from 
the above construction. Any scaled and translated copy of 
$\tau$ will have the values of $c$ and $r$ scaled and translated accordingly. 

We now embed the graph $G$ recursively, by embedding the outer face as an 
arbitrary triangle from $R'$, and then 
recursively embedding each tripod into the appropriate face by the procedure 
described above. Since we only insert tripods into balanced faces, it is easily 
seen that every tripod is being embedded inside a triangle of $R'$.

Overall, the construction uses at most $|S'|=\cO(\Delta^3)$ distinct slopes for the balanced 
edges, and at most $|R'||P'|=\cO(\Delta^5)$ distinct slopes for the tilted edges. The 
total number of slopes is then $\cO(\Delta^5)$, as claimed. 
\end{proof}

In the rest of this section, we prove the Tripod Drawing Lemma. 
Let $T$ be a labelled tripod and let $B$ be a bubble of~$T$.
Recall that the root edge of $B$ is the edge that belongs to a spine of~$T$. Note that the same root edge is shared by two 
bubbles of $T$. Recall also that a bubble is called trivial if it only has two vertices.

We now introduce some terminology that will be convenient for our description 
of the structure of a given bubble. 

\begin{definition}\label{def-dual}
Let $B$ be a nontrivial bubble in a tripod $T$. The unique internal face of 
$B$ adjacent to its root edge will be called \emph{the root face} of~$B$. The 
\emph{dual} of a bubble $B$ is the rooted binary tree $\wB$ whose nodes 
correspond bijectively to the internal faces of $B$, and two nodes are 
adjacent if and only if the corresponding faces of $B$ share an edge. The 
root of the tree $\wB$ is the node that represents the root face of~$B$. 

When dealing with the internal faces of $B$, we will employ the usual 
terminology of rooted trees; for instance, we say that a face $\Phi$ is the 
parent (or child) of a face $\Psi$ if the node representing $\Phi$ in 
$\wB$ is the parent (or child) of the node representing $\Psi$. For 
every internal face $\Phi$ of $B$, the three edges that form the boundary of $\Phi$ 
will be called \emph{the top edge}, \emph{the left edge} and \emph{the right 
edge}, where the top edge is the edge that $\Phi$ shares with its parent face 
(or the root edge, if $\Phi$ is the root face), while left and right edges are 
defined in such a way that the top, left, and right edge form a 
counterclockwise sequence on the boundary of $\Phi$. With this convention, we 
may speak of a left child face or right child face of $\Phi$ without any 
ambiguity. Our terminology is motivated by the usual convention of embedding 
rooted binary trees with their root on the top, and the parent, the left child 
and the right child appearing in counterclockwise order around every node of 
the tree. Furthermore, for a given face $\Phi$, the \emph{bottom vertex} of 
$\Phi$ is the common vertex of the left edge and right edge of $\Phi$.
\end{definition}  
                                                                   
Let us explicitly state the following simple fact which directly follows from 
our definitions.
\begin{observation}\label{obs-left}
Let $\Phi_1, \Phi_2,\dotsc, \Phi_k$ be a sequence of internal faces of a 
bubble $B$, such that for any $j<k$, $\Phi_{j+1}$ is the left child of 
$\Phi_j$. Then all the faces $\Phi_1,\dotsc,\Phi_k$ share a common vertex. In 
particular, if $B$ has maximum degree $\Delta$, then $k<\Delta$. An analogous 
observation holds for right children as well.
\end{observation}

We now describe an approach that allows us to embed an arbitrary bubble with 
maximum degree $\Delta$ inside a bounded area using a bounded number of 
slopes.                 

\begin{lemma}\label{lem-irelev}
Let $xyz$ be an equilateral triangle with vertex coordinates $x=(0,0)$, 
$y=(1,0)$ and $z=(1/2,-\sqrt 3/2)$. Fix two sequences of slopes 
$\alpha_1$, $\alpha_2$, \dots, $\alpha_{\Delta-1}$ and 
$\beta_1$, $\beta_2$, \dots, $\beta_{\Delta-1}$, with 
$0>\alpha_1>\alpha_2>\dotsb>\alpha_{\Delta-1}>-\pi/3$  and 
$0<\beta_1<\beta_2<\dotsb<\beta_{\Delta-1}<\pi/3$. Let $S$ be the set of 
$2\Delta-1$ slopes 
$\{0\}\cup\{\alpha_1,\alpha_2,\dotsc,\alpha_{\Delta-1}\}\cup\{\beta_1,\beta_2,\dotsc,\beta_{\Delta-1}\}$. 
Let $B$ be a bubble of maximum degree~$\Delta$. Then $B$ has a straight line 
embedding $\Em_B$ inside $xyz$ that only uses the slopes from the set $S$,  
the root edge of $\Em_B$ corresponds to the segment $xy$, 
and moreover the triangular faces of $\Em_B$ form
at most $2\Delta-3$ distinct triangles up to homothetic equivalence.

\end{lemma}
\begin{proof}
Proceed by induction on the size of $B$. If $B$ is trivial, the statement 
holds. Assume now that $B$ is a nontrivial bubble. Let $\Phi_0$ be the root 
edge of~$B$. See Figure~\ref{fig-irelev}. 

\begin{figure}
\hfil\includegraphics[scale=0.8]{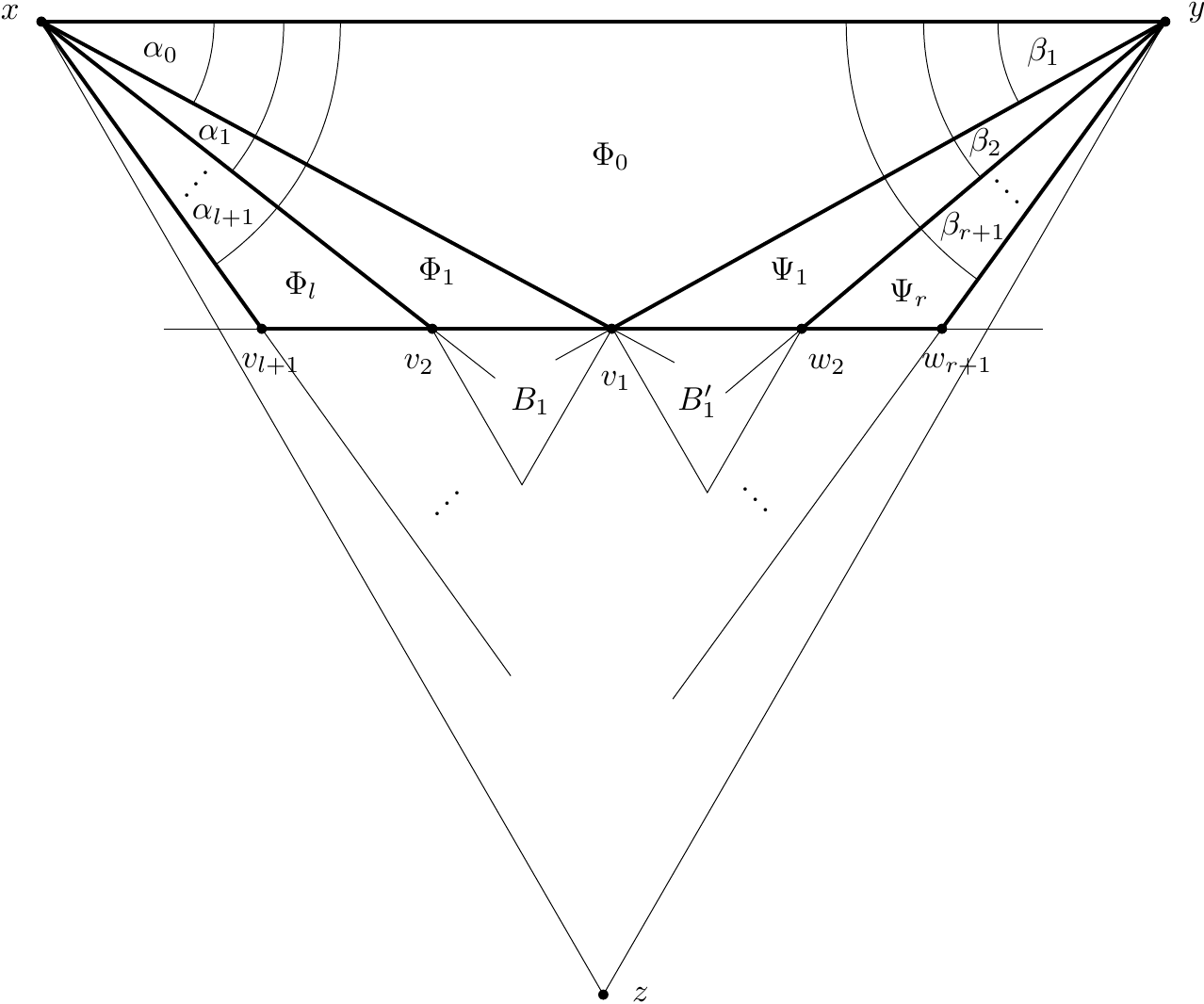}
\caption{Illustration of the proof of Lemma~\ref{lem-irelev}.}\label{fig-irelev}
\end{figure}

Define the maximal sequence of faces 
$\Phi_1,\Phi_2,\dotsc,\Phi_\ell$ in such a way that $\Phi_{i+1}$ is the left 
child of $\Phi_i$, with $\Phi_1$ being the left child of the root edge 
$\Phi_0$. The maximality of the sequence means that $\Phi_\ell$ has no left 
child. Symmetrically, define a maximal sequence of faces $\Psi_1,\dotsc, 
\Psi_r$ such that $\Psi_1$ is the right child of $\Phi_0$, and $\Psi_{i+1}$ 
is the right child of $\Psi_i$. By Observation~\ref{obs-left}, we know that 
$\ell<\Delta-1$ and $r<\Delta-1$.

Let $(p,\alpha)$ denote the ray starting at a point $p$ and heading 
in direction~$\alpha$. 

Let $B$ be an arbitrary bubble. Let $v_1$ be the intersection of the rays 
$(x,\alpha_1)$ and $(y,\beta_1)$. The root face $\Phi_0$ will be embedded as 
the triangle $xyv_1$. Define points $v_2,\dotsc,v_{\ell+1}$ by specifying 
$v_i$ as the intersection of $(x,\alpha_i)$ and $(v_1,\pi)$. The face $\Phi_i$ 
is then embedded as the triangle $xv_iv_{i+1}$. Similarly, define points 
$w_2,\dotsc,w_{r+1}$ where $w_i$ is the intersection of $(y,\beta_i)$ with 
$(v_1,0)$. Then $\Psi_1$ is embedded as the triangle $yv_1w_2$, while for 
$k>1$ we embed $\Psi_k$ as the triangle $yw_kw_{k+1}$. 

Note that when we remove the two vertices incident to the root edge from the 
bubble $B$, the remaining edges and vertices form a union of $\ell+r$ bubbles 
$B_1\cup\dotsb\cup B_\ell\cup B'_1\cup\dotsc\cup B'_\ell$, where $B_i$ is a 
bubble whose root edge is the right edge of $\Phi_i$ while $B'_j$ is rooted at 
the left edge of~$\Psi_j$. Using induction, we know that each $B_i$ has a 
straight line embedding inside the equilateral triangle whose top edge is the 
horizontal segment $v_iv_{i+1}$ (and symmetrically for $B'_j$).

This completes the proof.
\end{proof}  

\begin{corollary}\label{cor-irelev}
Let $xyz$ be an arbitrary triangle and $B$ a bubble of maximum degree $\Delta$. 
There are sets $S$ of $2\Delta-1$ slopes and $R$ of $2\Delta-3$ triangles that depend on $xyz$ but not on $B$, 
such that $B$ can be embedded inside $xyz$ using only slopes from $S$
and triangles from $R$ for triangular faces, in such 
a way that the root edge of $B$ coincides with the segment $xy$.
\end{corollary}                                                 
\begin{proof}
This follows from Lemma~\ref{lem-irelev}, using the fact that for any triangle 
there is an affine transform that maps it to an equilateral triangle, and that 
affine transforms preserve the number of distinct slopes used in a straight-line 
embedding.
\end{proof}
           
The construction from Lemma~\ref{lem-irelev} can be applied to embed all the 
irrelevant bubbles of a given labelled tripod $T$. Unfortunately, the 
construction of Lemma~\ref{lem-irelev} is not suitable for the embedding of 
relevant bubbles, because it provides no control about the position of the 
relevant vertices. Indeed, inside the triangle $xyz$ of the previous lemma, 
there are infinitely many points where a vertex may be embedded by the 
construction described in the proof of the lemma. Thus, we can give no upper 
bound on the number of potential embeddings of relevant vertices.

For this reason, we now describe a more complicated embedding procedure, which 
allows us to control the position of the relevant vertices. We first need some 
auxiliary definitions. 

\begin{definition}
An \emph{adder} $A$ is a bubble with a root edge $h$ and another edge $t\neq 
h$, such that the dual tree of $A$ is a path, and the edge $t$ is an external 
edge adjacent to the single leaf face of $A$. See Figure~\ref{fig-adder}. The 
edges $h$ and $t$ are called \emph{head} and \emph{tail} of the adder. It is 
easy to see that every adder contains a unique path $Z$ whose first edge is 
$h$, its last edge is $t$ and no other edge of $Z$ belongs to the outer face 
of $A$. The path $Z$ will be called the \emph{zigzag path} of the adder $A$. 
The \emph{length} of the adder is defined to be the number of edges of its 
zigzag path. By definition, each adder has length at least 2. An adder of 
length 2 will be called \emph{degenerate}.  
\end{definition}                                                         

\begin{figure}
\hfil\includegraphics[scale=1]{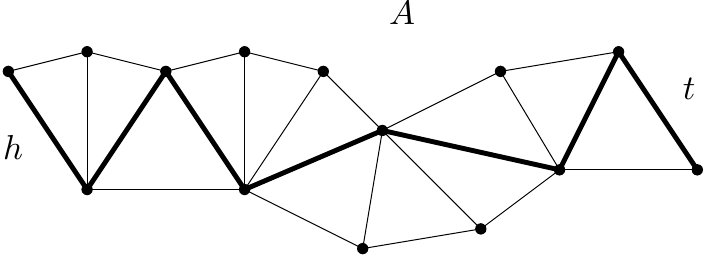}
\caption{An adder. The bold edges form the zigzag path.}\label{fig-adder}
\end{figure}

We will now show that adders of bounded degree can be embedded inside a 
prescribed quadrilateral using a bounded number of slopes and triangles.

\begin{lemma}\label{lem-adder}
For every convex quadrilateral $Q=abcd$ and for every $\Delta$ there is a set 
$S$ of ${\cal O}(\Delta)$ slopes, a set $S_0\subseteq S$ of ${\cal O}(1)$ 
slopes, and a set $R$ of ${\cal O}(\Delta)$ triangles such that any 
nondegenerate adder $A$ of maximum degree $\Delta$ has a straight line 
embedding $\Em_A$ with the following properties:
\begin{enumerate}
\item All the edge-slopes of $\Em_A$ belong to the set~$S$.
\item All the edges on the outer face of $\Em_A$ have slopes from the set~$S_0$.
\item Each internal face of $\Em_A$ is homothetic to a triangle from~$R$. 
\item The head of $A$ coincides the edge $ab$ of $Q$ and the tail of $A$ 
coincides with~$cd$.
\item The embedding $\Em_A$ is contained in the convex hull of $abcd$.
\end{enumerate}
\end{lemma}                                 
\begin{proof}  
Note that the lemma is clearly true when restricted to adders of length at 
most four (or any other bounded length). In the rest of the proof, we assume 
that $A$ is an adder of length at least five.

We first deal with the case when the edges $ab$ and $cd$ are parallel (i.e., 
$Q$ is a trapezoid), and the adder $A$ has odd length $\ell=2k+1$. Without 
loss of generality, assume that the segments $ab$ and $cd$ are horizontal and 
that the line containing $cd$ is above the line containing $ab$. Let $\alpha$ 
be the slope of the diagonal $ac$ and $\beta$ the slope of the diagonal $bd$, 
with $0<\alpha<\beta<\pi$. Let $e$ be the point where the two diagonals 
intersect. Notice that the two triangles $abe$ and $cde$ are homothetic. Let 
$r=\|ab\|/\|cd\|=\|ae\|/\|ce\|$ be the dilation factor of the homothecy.

Let $Z$ be the zigzag path of $A$. Let us identify the head of $A$ with the 
segment $ab$ and the tail of $A$ with $cd$, in such a way that the cyclic 
order of the four points $abcd$ on the boundary of $Q$ is the same as the 
cyclic order in which the corresponding vertices appear on the outer face of 
$A$. 
           
\begin{figure}
\hfil\includegraphics[scale=0.75]{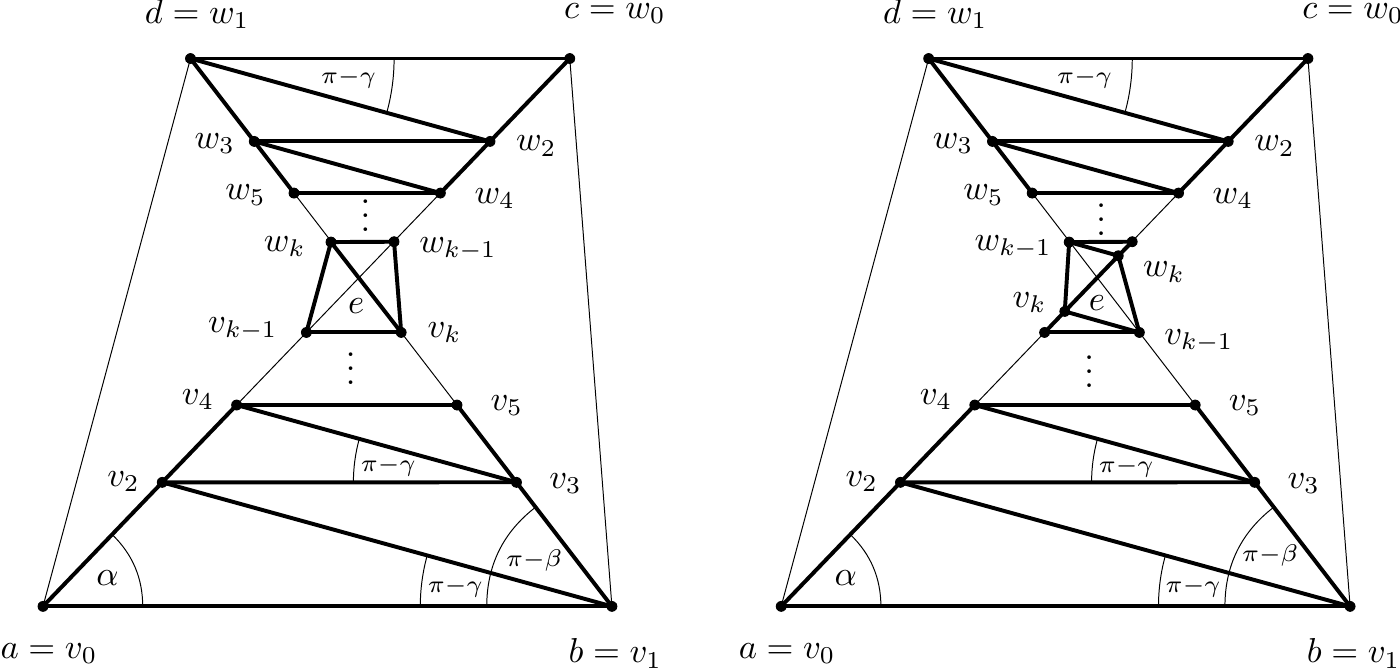} \caption{Embedding an adder 
with prescribed head and tail. These figures illustrates the embedding of the 
adder of odd length $2k+1$. The two figures correspond to the two cases 
depending on the parity of $k$.}\label{fig-embadder}
\end{figure}

Since $A$ has odd length, the endpoints of its zigzag path are diagonally 
opposite in $Q$, see Figure~\ref{fig-embadder}. We 
lose no generality by assuming that $a$ and $c$ are the endpoints of the 
zigzag path. Let $v_0,v_1,v_2,\dotsc,v_k,w_k,w_{k-1},w_{k_2},\dotsc,w_1,w_0$ 
be the sequence of the vertices of $Z$, in the order in which they appear on 
the path $Z$, with $v_0=a$, $v_1=b$, $w_0=c$, and $w_1=d$. Fix an arbitrary 
slope $\gamma$ such that $\beta<\gamma<\pi$. All the vertices of $Z$ will be 
embedded on the two diagonals $ac$ and $bd$. Since the first two and last two 
vertices have already been embedded, let us proceed by induction, separately 
in each half of $Z$. If, for some $i\ge0$, the vertex $v_i$ has already been 
embedded on the diagonal $ac$, then we embed $v_{i+1}$ on $bd$ in such a way 
that the segment $v_iv_{i+1}$ is horizontal. If $v_i$ has been embedded on 
the diagonal $bd$, then $v_{i+1}$ is embedded on $ac$ and the slope of 
$v_iv_{i+1}$ is equal to~$\gamma$.

We proceed similarly with the vertices $w_i$: if $w_i$ is on $ac$ then 
$w_{i+1}$ is on $bd$ and the segment $w_iw_{i+1}$ has slope $\gamma$; 
otherwise $w_i$ is on $bd$ and $w_{i+1}$ is on $ac$ and the corresponding 
segment is horizontal.

We may easily show by induction that for any $i$, the triangles $ev_iv_{i+1}$ 
and $ew_iw_{i+1}$ are similar, all of them with the same ratio 
$r=\|ev_i\|/\|ew_i\|$. Furthermore, we see that $ev_iv_{i+1}$ is similar to 
$ev_{i+2}v_{i+3}$, with a ratio $q$ that is independent of~$i$. From these 
facts, we see that all the segments of the form $v_iw_{i+1}$ have at most two 
distinct slopes (depending on the parity of $i$), and similarly for the 
segments of the form $w_iv_{i+1}$.

Let us consider all the triangles formed by triples of vertices $xyz$ where 
$x,y$ and $z$ are three consecutive vertices of the path $Z$. Note that these 
triangles are internally disjoint, and their edges form at most six distinct 
slopes, namely $0,\alpha,\beta,\gamma$, the slope of the segment $v_kw_{k-1}$ 
and the slope of the segment $v_{k-1}w_k$. Furthermore, the latter two slopes 
belong to a set of at most four slopes that are independent of $k$, and hence 
independent of the adder $A$. The union of the above-described triangles will 
form the outer boundary of our embedding of $A$. It remains to place the 
vertices of $A$ that do not belong to $Z$ to this boundary.

Let us fix $\Delta-2$ additional slopes 
$\gamma_1<\gamma_2<\dotsb<\gamma_{\Delta-2}$ which are all greater than 
$\gamma$ but smaller than $\pi$. Note than any vertex $u$ of $A$ that does not 
belong to $Z$ is incident to exactly one edge that does not belong to the 
outer face of $A$, and this edge connects $u$ to a vertex of $Z$. Thus, to 
complete the description of the embedding of $A$, it suffices to specify, for 
every vertex $v$ of $Z$, the slopes of all the edges that do not belong to the 
outer face of $A$ and that connect $v$ to a vertex not belonging to $Z$. Thus, 
let us fix an arbitrary vertex $v$ of $Z$. Let us assume that $v$ has been 
embedded on the diagonal $ac$ and that $v=v_i$ for some $i\le k$ (the cases 
when $v$ belongs to $bd$ or $v=w_i$ are analogous). Let $u_1,\dotsc,u_\ell$ be 
the vertices not belonging to $Z$ and adjacent to $v$ by an internal edge of 
$A$. Note that if $v$ has at least one such neighbor $u_i$, then $v\neq v_1$, 
because $v_1$ is not incident to any edge not belonging to the outer face. Let 
$v^+$ be the vertex that follows after $v$ on $Z$ (typically, $v^+=v_{i+1}$, 
unless $v=v_k$, when $v^+=w_k$). Assume that the vertices $u_1,\dotsc u_\ell$ 
are listed in their counterclockwise order with respect to the neighborhood 
of~$v$. Let us place each $u_i$ at the intersection of the line $v_{i-1}v^+$ 
and the ray $(v,\pi+\gamma_i)$. This choice guarantees that the edge $vu_i$ 
has slope $\gamma_i$.

We have thus found a straight line embedding of $A$ that has all the required 
properties and uses at most $\Delta +{\cal O}(1)$ slopes. This completes the 
case when $A$ is an odd-length adder and $Q$ is a trapezoid.

Assume now that $A$ is an arbitrary nondegenerate adder of length $\ell\ge 
5$, and $Q$ is an arbitrary convex quadrilateral. Our goal is to reduce this 
situation to the cases solved above. Note that the adder $A$ can be written as 
a union of two non-degenerate sub-adders $A_1$ and $A_2$, where $A_1$ has odd 
length, $A_2$ has length three or four, $A_1$ has the same head as $A$, $A_2$ 
has the same tail as $A$, the tail of $A_1$ is the head of $A_2$, and the 
adders $A_1$ and $A_2$ are otherwise disjoint. Accordingly, the convex 
quadrilateral $Q=abcd$ can be decomposed into a union of two internally 
disjoint quadrilaterals $Q_1=abc'd'$ and $Q_2=d'c'cd$, where $Q_1$ is a 
trapezoid. We may now use our previous arguments to construct an embedding of 
$A_1$ inside $Q_1$, and an embedding of $A_2$ inside $Q_2$, and combine the
two embeddings into an embedding of $Q$ satisfying the conditions of the 
lemma.
\end{proof}                                        

We will use adders as basic building blocks in a procedure that embeds any 
given bubble with prescribed relevant vertices in such a way that the 
embedding of all the relevant vertices is chosen from a finite set of 
points. The following technical lemma summarizes all the key properties of the 
bubble embedding that we are about to construct.

\begin{lemma}\label{lem-bubble}   
Let $T=abc$ be an isosceles triangle with base $ab$, and with internal angles 
$\varepsilon/2, \varepsilon/2$ and $\pi-\varepsilon$. Assume that the line $ab$ 
is horizontal and the point $c$ is below the line~$ab$. For every $\Delta>0$ 
there is a set $S$ of $\cO(\Delta^{3})$ slopes, a set $P$ of $\cO(\Delta)$ 
points, and a set $R$ of $\cO(\Delta^{3})$ triangles, such that every labelled 
bubble $B\in\B(\Delta)$ has an embedding $\Em_B$ with the following properties.
\begin{enumerate}
\item All the edge-slopes of $\Em_B$ belong to $S$.
\item Any relevant vertex of $B$ is embedded at a point from $P$.
\item Every internal face of $\Em_B$ is homothetic to a triangle from $R$.
\item
The root edge of $B$ coincides with the segment $ab$.     
\item                                                      
The whole embedding $\Em_B$ is inside the triangle~$T$.
\item 
Any relevant vertex of $\Em_B$ has visibility in any direction from the set
$\langle\pi+\varepsilon, 2\pi-\varepsilon\rangle$.
\end{enumerate}
\end{lemma}                                            
\begin{proof}
Let us first introduce some terminology (see Figure~\ref{fig-blechy}). Let
$B\in\B(\Delta)$ be a labelled 
bubble. Recall from Definition~\ref{def-dual} that the dual of $B$, denoted 
by $\wB$, is a rooted binary tree whose root corresponds to the root face 
of~$B$. For an internal face $\Phi$ of $B$, we let $\wf$ denote the 
corresponding node of~$\wB$. We distinguish several types of nodes in~$\wB$. 
A node $\wf$ is called $\emph{relevant node}$, if the bottom vertex of 
the face $\Phi$ is a relevant vertex of~$B$. A node $\wf$ of $\wB$ is called 
\emph{peripheral} if the subtree of $\wB$ rooted at $\wf$ does not contain 
any relevant node, in other words, neither $\wf$ nor any descendant of $\wf$ 
is relevant. A node is \emph{central} if it is not peripheral. Note that the 
central nodes induce a subtree of $\wB$; we let $\wB'$ denote this subtree. 
By construction, all the leaves of $\wB'$ are relevant nodes (but there may 
be relevant nodes that are not leaves). 

\begin{figure}
 \includegraphics{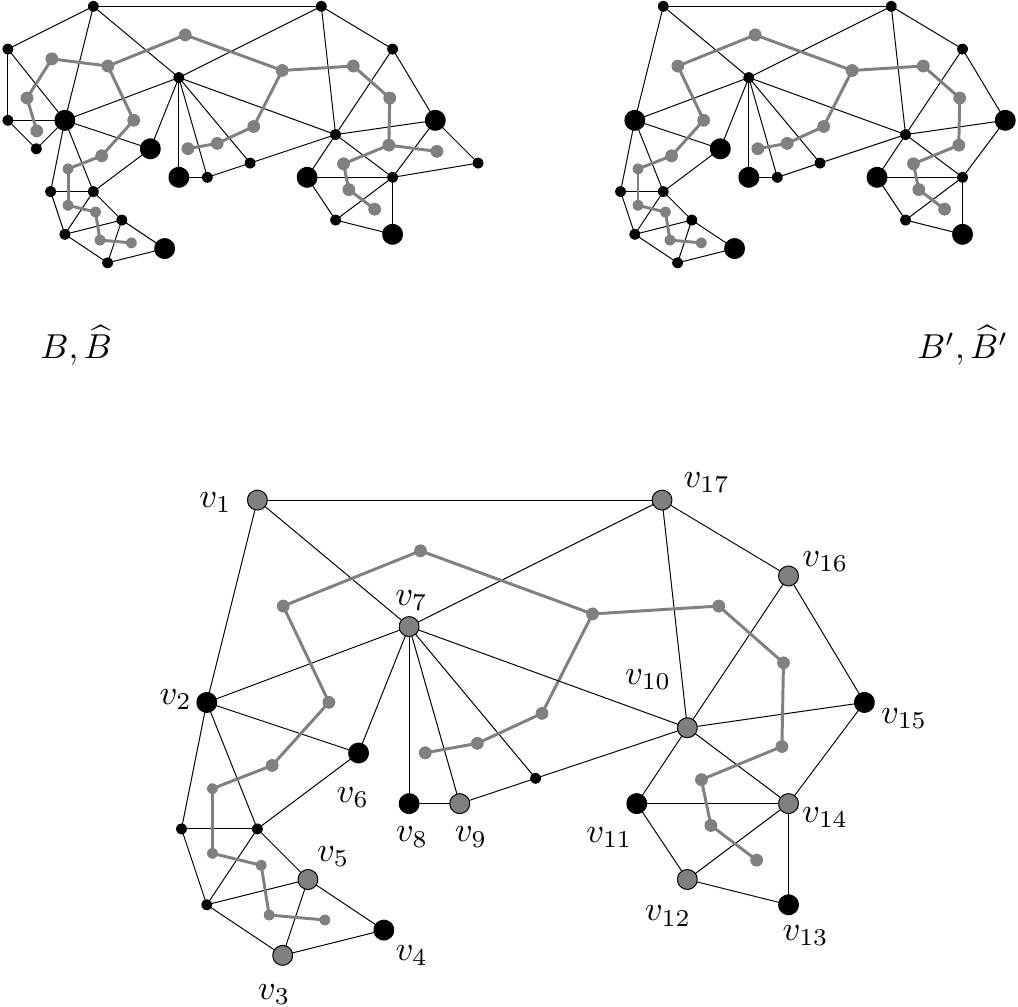}
\caption{An example of a labelled bubble $B$ with its dual tree $\wB$.
Relevant vertices are represented by large black disks. The large gray disks of
the bottom figure represent the non-relevant priority
vertices.}\label{fig-blechy}
\end{figure}

A node $\wf$ of $\wB'$ is a \emph{branching node} if both its children 
belong to~$\wB'$ as well. A node of $\wB'$ is a \emph{connecting node} if 
it is neither relevant nor branching. By definition, each connecting node has 
a unique child in~$\wB'$, and the connecting nodes induce in $\wB'$ a 
disjoint union of paths. We call these paths \emph{the connections}.
              
We say that a face $\Phi$ of $B$ is al \emph{relevant face} if the corresponding node 
$\wf$ is a relevant node. Peripheral faces, branching faces and connecting 
faces are defined analogously. Let $B'$ be the subgraph of $B$ whose dual 
is~$\wB'$. If $\wB'$ is empty, define $B'$ to be the trivial bubble 
consisting of the root edge of~$B$. In any case, $B'$ is a subbubble of $B$ and has the same
root edge as~$B$.

Note that since every leaf of $\wB'$ is a relevant node, and since $B$ has 
at most $3\Delta$ relevant vertices by definition of $\B(\Delta)$, the tree 
$\wB'$ has at most $\cO(\Delta)$ leaves and consequently at most 
$\cO(\Delta)$ branching nodes.

Let us now describe the basic idea of the proof. We begin by specifying the 
set~$P$ of points. The points of $P$ will form a convex cup inside the 
triangle~$T$. For a given bubble $B\in\B(\Delta)$, we construct the embedding 
$\Em_B$ in three steps. In the first step, we take all the vertices of $B$ 
that belong to relevant faces and branching faces, and embed them to the 
points of~$P$. In the second step, we embed all the connecting faces. Each 
connection in $\wB'$ corresponds to a (possibly degenerate) adder contained in~$B'$, 
whose head and tail have been embedded in the first step. Using the 
construction from Lemma~\ref{lem-adder}, we insert these adders into the 
embedding. Thus, in the first two steps, we construct an embedding of $B'$. 
In the third step, we extend this embedding into an embedding of $B$ by 
adding the peripheral faces. These faces form a disjoint union of 
subbubbles, each of them rooted at an edge belonging to the outer face 
of~$B'$. We use Corollary~\ref{cor-irelev} to embed each of these subbubbles 
into a thin triangle above a given root edge.
                                                
Let us describe the individual steps in detail. Set $D=18\Delta$. Recall 
that $T$ is an isosceles triangle with base $ab$. Let $C$ be any circular arc 
with endpoints $a$ and $b$, drawn inside~$T$. Choose a sequence $p_1, p_2, 
\dotsc, p_D$ of distinct points of $C$, in such a way that $p_1=a$, $p_D=b$, 
and the remaining points are chosen arbitrarily on $C$ in order to form a 
left-to-right sequence. Let $P$ be the set $\{p_1,\dotsc,p_D\}$. 
                                                     
Let us say that a vertex $v$ of $B$ is a \emph{priority vertex} if it either 
belongs to a relevant face, or it belongs to a branching face, or it belongs 
to the root edge of~$B$. Note that all priority vertices actually belong 
to~$B'$, and that each relevant vertex is a priority vertex as well. Let $\ell$ 
be the number of priority vertices. We know 
that $B$ has at most $3\Delta$ relevant faces. Since every leaf of $\wB'$ 
represents a relevant face, we see that $B'$ has at most $3\Delta-1$ branching 
faces. This implies that $\ell<D=18\Delta$. 

Let $v_1,v_2,\dotsc,v_\ell$ be the sequence of all the priority vertices of 
$B$, listed in counterclockwise order of their appearance on the outer face 
of $B$, in such a way that $v_1$ and $v_\ell$ are the vertices of the root 
edge of~$B$. For each $i\in\{1,\dotsc,\ell-1\}$, we embed the vertex $v_i$ 
on the point $p_i$, while the vertex $v_\ell$ is embedded on the point 
$v_D=b$. Note that this embedding guarantees that the root edge of $B$ 
coincides with the segment $ab=p_1 p_D$. Moreover, since this embedding 
preserves the cyclic order of the vertices on the boundary of the outer face, 
we know that the edges induced by the priority vertices do not cross.
This completes the first step of the embedding.
                                                                     
In the second step, we describe the embedding of the connecting faces of~$B$. 
Let $\Phi_1, \Phi_2,\dotsc,\Phi_k$ be a sequence of faces of $B$ 
corresponding to a connection in $\wB$, where we assume that for each $i<k$, 
the node $\wf_i$ is the parent of $\wf_{i+1}$ in~$\wB$. See
Figure~\ref{fig-adderA}.
Let $x$ be the left vertex of $\Phi_1$ and let $y$ be the right vertex 
of~$\Phi_1$. The vertices $x$ and $y$ either form the root edge of $B$, or 
they belong to the parent face of $\Phi_1$, which is either a relevant face 
or a branching face. In either case, both $x$ and $y$ are priority vertices. 
In particular, $x$ corresponds to a point $p_m\in P$, and $y$ corresponds to $p_n\in 
P$, for some $m<n$.

\begin{figure}
\hfil\includegraphics{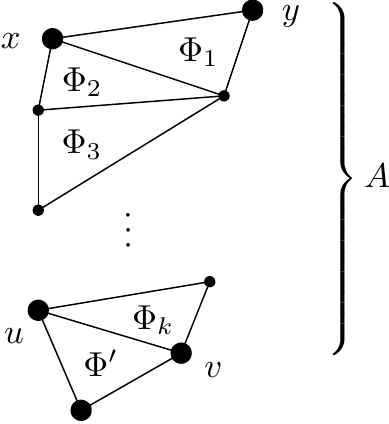}
\caption{An adder representing a connection in $\wB$.}\label{fig-adderA}
\end{figure}

Consider now the face $\Phi_k$. Since it is neither relevant nor branching, 
it has a unique child face $\Phi'$ in~$B'$. The face $\Phi'$ is relevant or 
branching, so all its vertices are priority vertices. 
Let $u$ be the left vertex of $\Phi'$ and let $v$ be its right vertex. The 
edge $uv$ is the intersection of $\Phi'$ and $\Phi_k$. Let $A$ be the adder 
formed by the union of the faces $\Phi_1,\dotsc,\Phi_k$, with head $xy$ and 
tail $uv$. Note that this adder does not contain any other priority vertices 
apart from $x$, $y$, $u$ and~$v$. In particular, the vertex $u$ is either 
equal to $x$, or it corresponds to $p_{m+1}$. For the vertex $v$, we have 
three possibilities: either $v=y$, or $v=p_{n-1}$, or $v=p_{\ell-1}$ and  
$y=p_D$.

Let us first deal with the case when the adder $A$ is degenerate, i.e., 
either $x=u$ or $y=v$. We first define a set $Q$ of auxiliary points (see
Figure.~\ref{fig-points-q}. For 
every $i<D$, consider the segment $p_ip_{i+1}$, and subdivide this segment 
with $\Delta-2$ new points $q^i_1,q^i_2,\dotsc,q^i_{\Delta-2}$. Next, for 
$i<D-1$, consider also the segment $p_ip_D$ and subdivide it with $\Delta-2$ 
points $\tilde q^i_1,\tilde q^i_2,\dotsc,\tilde q^i_{\Delta-2}$. Let $Q$ be 
the set of all the points $q_j^i$ and $\tilde q_j^i$, for all $i$ and~$j$. 

\begin{figure}
\hfil\includegraphics{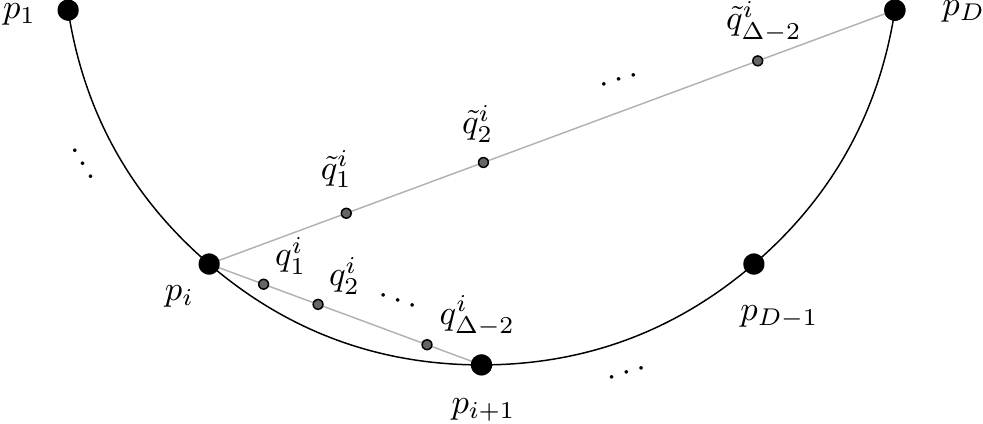}
\caption{The auxiliary points from the set $Q$.}\label{fig-points-q}
\end{figure}

Assume now that $A$ is a degenerate adder with $x=u$ (the case when $y=v$ is 
analogous). Recall that $A$ has $k$ internal faces $\Phi_1,\dotsc,\Phi_k$. 
All these faces share the vertex $x$, and in particular, $x$ has degree $k+1$ 
in~$A$. This shows that $k<\Delta$, and consequently there are at most 
$\Delta -2$ non-priority vertices in $A$, all of them on a path from $v$ 
to~$y$. See Figure~\ref{fig-degenerate}. If $v=p_{n-1}$, we embed these
non-priority vertices on the points 
$q_1^{n-1},\dotsc, q_{k-1}^{n-1}$. On the other hand, if $v=p_{\ell-1}$ and 
$y=p_D$, we embed the non-priority vertices of $A$ on the points $\tilde 
q_1^{\ell-1},\dotsc, \tilde q_{k-1}^{\ell-1}$. This determines the embedding of~$A$.

\begin{figure}
\includegraphics[width=\textwidth]{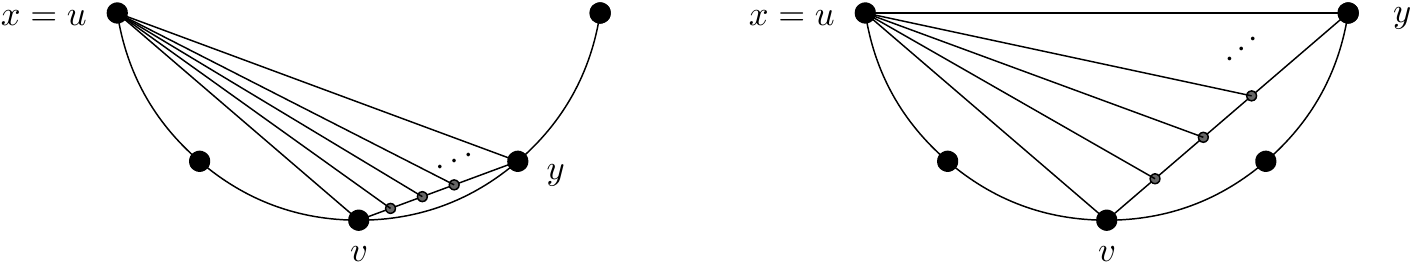}
\caption{The embedding of a degenerate connection adder.}\label{fig-degenerate}
\end{figure}

Consider now the case when $A$ is non-degenerate. The four vertices $x$, $y$, 
$u$ and $v$ form a convex quadrilateral, and we embed $A$ inside this quadrilateral,
using the construction of Lemma~\ref{lem-adder}. This again determines the 
embedding of~$A$.

Using the constructions described above, we embed all the adders representing connections in~$\wB$. Note that each adder is embedded inside the convex hull of its head 
and tail. Moreover, if $A$ and $A'$ are adders representing two different 
connections, the convex hull of the head and tail of $A$ is disjoint from the 
convex hull of the head and tail of $A'$, except for at most one vertex 
shared by the two adders. This shows that the embedding is indeed a plane 
embedding of the graph~$B'$, completing the second step of the construction. 

Before we describe the last step, let us estimate the number of vertices, 
edge-slopes and internal faces that may arise in the first two steps. 
Clearly, any relevant vertex is embedded on a point from the set $P$, which 
has size $\cO(\Delta)$ and does not depend on the bubble~$B$.

Any edge $e$ embedded in the first two steps may have one of the following 
forms.
\begin{itemize}
\item The edge $e$ connects two points from $P$. Such edges can take at most 
$\cO(\Delta^2)$ slopes.
\item The edge $e$ connects a vertex from $P$ to a vertex from $Q$. This yields 
$\cO(\Delta^3)$ possible slopes.
\item The edge $e$ connects two vertices of $Q$. This is only possible when 
both vertices of $e$ belong to a segment determined by a pair of points 
in~$P$. The slope of $e$ is then equal to a slope determined by two points 
from~$P$.
\item 
The edge $e$ belongs to a non-degenerate adder $A$ representing a connection 
in~$B$. In the embedding from Lemma~\ref{lem-adder}, the edges of a given adder $A$ determine at most 
$\cO(\Delta)$ slopes, and these slopes only depend on the four 
vertices forming the head and tail of~$A$. This fourtuple of vertices has 
the form $\{p_i,p_{i+1},p_{j-1},p_j\}$ or $\{p_i,p_{i+1},p_{j-1},p_D\}$. There 
are $\cO(\Delta^2)$ such fourtuples and hence $\cO(\Delta^3)$ possible slopes 
for the edges of this type.
\end{itemize}             
Overall, there is a set of $\cO(\Delta^3)$ slopes, independent of $B$, such 
that any edge embedded in the first two steps has one of these slopes.

Next, we count homothecy types of internal faces. Any internal face $\Phi$ 
embedded in the first two steps has one of the following types.
\begin{itemize}
\item
All the vertices of $\Phi$ belong to $P$. There are $\cO(\Delta^3)$ such faces. 
\item 
$\Phi$ has two vertices from $P$ and one vertex from $Q$. In such case the 
triple of vertices of $\Phi$ must be of one of these forms, for some values of 
$i$, $j$ and $k$: $\{p_i,p_j,q^j_1\}$, or $\{p_i,p_j,q^{j-1}_k\}$, or 
$\{p_i,p_D,\tilde q^j_k\}$. There are $\cO(\Delta^3)$ such triples.
\item                             
$\Phi$ has two vertices from $Q$ and one vertex from~$P$. In such case the 
two vertices from $Q$ are of the form $\{q_j^i, q_{j+1}^i\}$ or $\{\tilde 
q_j^i, \tilde q_{j+1}^i\}$ for some $i$ and $j$. This again gives 
$\cO(\Delta^3)$ possibilities for~$\Phi$.
\item                     
$\Phi$ is an internal face of a non-degenerate adder, embedded by 
Lemma~\ref{lem-adder}. Lemma~\ref{lem-adder} shows that the internal faces of 
such an adder form $\cO(\Delta)$ homothecy types depending only on the 
position of head and tail. Since there are $\cO(\Delta^2)$ positions for head 
and tail, this gives $\cO(\Delta^3)$ triangle types up to homothecy.
\end{itemize}                                      
We conclude that each internal face of $B'$ is homothetic to one of $\cO(\Delta^3)$ 
triangles, and these triangles do not depend on~$B'$.

We next estimate the number of slopes formed by edges on the outer face 
of~$B'$. For $e$ on the outer face of~$B'$ there are two possibilities.
\begin{itemize}
\item 
If both endpoints of $e$ are priority vertices, or if $e$ belongs to a 
connection represented by a degenerate adder, then the line determined by the 
segment~$e$ passes through two points of~$P$. In particular, such a segment 
$e$ must have one of $\cO(\Delta^2)$ slopes determined by~$P$.
\item                                                         
Suppose $e$ belongs to the outer face of a non-degenerate adder $A$. By 
Lemma~\ref{lem-adder}, the edges of the outer face of $A$ have $\cO(1)$ 
distinct slopes, depending on the head and tail of~$A$. Overall, such edges 
have at most $\cO(\Delta^2)$ slopes.
\end{itemize}
This shows that the slopes of the edges of the outer face of $B'$ all belong 
to a set of $\cO(\Delta^2)$ slopes.

To finish the proof, it remains to perform the third step of the 
construction, where we embed the peripheral faces. Fix an angle $\delta>0$ 
such that $\delta<\varepsilon/2$ and any two distinct edge-slopes used in the 
first two steps of the construction differ by more than $2\delta$. Let $e$ be 
an edge of the outer face of~$B'$. Let $T_e$ be an isosceles triangle whose 
base is the edge $e$, whose internal angles have size $\delta$, $\delta$, and 
$\pi-2\delta$, and which lies in the outer face of~$B'$. It is easy to check 
that our choice of $\delta$ guarantees that for any two edges $e$ and $f$ on 
the outer face of $B'$, the triangles $T_e$ and $T_f$ are disjoint, except 
for a possible common vertex of $e$ and~$f$.

Let $\wB_0$ be a maximal subtree of $\wB$ formed entirely by peripheral 
nodes, and let $B_0$ be the dual of~$\wB_0$. Note that $B_0$ is a subbubble 
of $B$ rooted at an edge of the outer face of~$B'$. Let $e$ be the root edge 
of~$B_0$. Using Corollary~\ref{cor-irelev}, we embed $B_0$ inside $T_e$, in 
such a way that the root edge of $B_0$ coincides with~$e$. This embedding of 
$B_0$ uses $\cO(\Delta)$ edge-slopes and $\cO(\Delta)$ triangle types for its 
internal faces, and these edge-slopes and triangle types only depend on the 
slope of~$e$. 

Since the edges on the outer face of $B'$ may have at most $\cO(\Delta^2)$ 
edge-slopes, we may embed all the peripheral faces of $B$, while using only 
$\cO(\Delta^3)$ edge-slopes and $\cO(\Delta^3)$ triangle types in addition to 
the edge-slopes and triangle types used in the first two steps of the construction.

This completes the last step of the construction. It is easy to check that in 
the obtained embedding of $B$, any relevant vertex has visibility in any 
direction from the set $\langle \pi+\varepsilon, 2\pi-\varepsilon\rangle$, and 
the remaining claims of the lemma have already been verified.
\end{proof}
            
At last, we are ready to give the proof of the Tripod Drawing Lemma.
Let us recall its statement:

\begin{lemmaagain}   
For every $\Delta$ there is a set of slopes $S$ of size $\cO(\Delta^3)$, a set 
of points $P$ of size $\cO(\Delta^2)$, and a set of triangles $R$ of size 
$\cO(\Delta^3)$, such that every labelled tripod $T\in\Tr(\Delta)$ has a 
straight-line embedding $\Em_T$ with the following properties:
\begin{enumerate}
\item 
The slope of any edge in the embedding $\Em_T$ belongs to~$S$.
\item 
Each relevant vertex of $\Em_T$ is embedded on a point from~$P$.
\item 
Each internal face of $\Em_T$ is homothetic to a triangle from~$R$.
\item                                                           
The central vertex of $\Em_T$ is embedded in the origin of the 
plane.
\item 
Any vertex of $\Em_T$ is embedded at a distance at most $1$ from the origin.
\item
Each spine of $T$ is embedded on a single ray starting from the origin. The three 
rays containing the spines have directed slopes $0$, $2\pi/3$ and~$4\pi/3$. 
Let these three rays be denoted by $r_1$, $r_2$ and $r_3$, respectively.  
\item      
Let $\uhel{i}{j}$ denote the closed convex region whose boundary is formed by 
the rays $r_i$ and~$r_j$. Any relevant vertex of $\Em_T$ embedded in the 
region $\uhel{1}{2}$ (or $\uhel{2}{3}$, or $\uhel{1}{3}$) has visibility in 
any direction from the set $\langle \varepsilon, 2\pi/3-\varepsilon\rangle$ 
(or $\langle 2\pi/3+\varepsilon, 4\pi/3-\varepsilon\rangle$, or $\langle 
4\pi/3+\varepsilon,2\pi-\varepsilon\rangle$, respectively). 

Note that the three regions $\uhel{1}{2}$, $\uhel{2}{3}$ and $\uhel{1}{3}$ 
are not disjoint. For instance, if a relevant vertex of $T$ is embedded 
on the ray $r_1$, it belongs to both $\uhel{1}{2}$ and $\uhel{1}{3}$, and 
hence it must have visibility in any direction from the set $\langle 
\varepsilon, 2\pi/3-\varepsilon\rangle\cup \langle 
4\pi/3+\varepsilon,2\pi-\varepsilon\rangle$. 
\end{enumerate}
\end{lemmaagain}                 
\begin{proof}
Fix a tripod $T\in\Tr(\Delta)$. Let $X$, $Y$, and $Z$ be the three legs of the 
tripod $T$. The center $c$ of the tripod will coincide with the origin of the 
coordinate system, and the spines of the three legs will be embedded onto 
three rays with slopes $0$, $2\pi/3$ and $4\pi/3$ starting at the origin. We 
will now describe how to embed the leg $X$ onto the horizontal ray $(c,0)$. 
The embeddings of the remaining two legs are then built by an analogous 
procedure, rotated by $2\pi/3$ and $4\pi/3$.

Let $X$ be a fixed leg of the tripod, represented as a sequence 
$D_1,D_2,\dotsc,D_k$ of double bubbles, ordered from the center outwards. 
Recall that a bubble is called relevant if it contains at least one relevant 
vertex. We will also say that a double bubble is relevant if at least one of 
its two parts is relevant. 
                                                     
Define a parameter $D$ by $D=13\Delta$.
The leg $X$ can have at most $6\Delta$ relevant double bubbles. A maximal 
consecutive sequence of the form $D_i,D_{i+1},\dotsc,D_j$ in which each 
element is an irrelevant double bubble will be called an \emph{irrelevant 
run}. We partition $X$ into a sequence of \emph{parts} 
$P_1,P_2,\dotsc,P_\ell$, where a part is either a single relevant double 
bubble, or a nonempty irrelevant run. Since by definition no two irrelevant 
runs are consecutive, we see that $X$ has at most $12\Delta+1<D$ parts.

Let $T_\varepsilon$ be an isosceles triangle with internal angles of size 
$\varepsilon/2$, $\varepsilon/2$ and $\pi-\varepsilon$ whose base edge is 
horizontal. From Lemma~\ref{lem-bubble}, we know that there is a set of 
points $P_\varepsilon\subset T_\varepsilon$ of size $\cO(\Delta)$, a 
set of slopes $S_\varepsilon$ of size $\cO(\Delta^3)$ and
set of triangles $R_\varepsilon$ of size $\cO(\Delta^3)$
such that any bubble 
of $B\in\B(\Delta)$ can be embedded inside $T_\varepsilon$ using slopes from 
$S_\varepsilon$ in such a way that each relevant vertex of $B$ coincides with 
a point from the set $P_\varepsilon$ and the internal faces of the embedding 
are homothetic to triangles in $R$. Let $\Em_B$ denote this embedding.

We will combine these embeddings to obtain an embedding of the whole leg $X$. 
To each of the at most $D$ parts of $X$ we will assign a segment of length 
$L=\frac{1}{D}$ on the horizontal ray $(c,0)$. 

Assume first that $P_i$ is a part of $X$ consisting of a single relevant 
double bubble, formed by a pair of bubbles $B$ and~$C$. We will embed $P_i$ 
in such a way that the common root edge of $B$ and $C$ coincides with a 
horizontal segment $e_i$ of length $L$, whose endpoints have horizontal 
coordinates $(i-1)L$ and $iL$. The two bubbles $B$ and $C$ are then embedded 
inside two scaled and translated copies of $T_\varepsilon$ that share a 
common base $e_i$, using the embeddings $\Em_B$ and $\Em_C$, possibly 
reflected along the horizontal axis.

Now assume that $P_i$ is a part of $X$ that consists of an irrelevant run of 
$k$ irrelevant double bubbles $D_j,D_{j+1},\dotsc, D_{i+k-1}$. We embed the 
root edge of each double bubble onto a segment of length $L/k$, and embed the 
rest of the double bubble into a scaled and translated copy of 
$T_\varepsilon$. We then concatenate these embeddings to obtain an embedding 
of the whole irrelevant run, which will occupy a segment of length exactly 
$L$ on the spine of~$X$.         

Overall, since the leg has at most $D$ parts, the whole leg will be 
embedded at distance at most 1 from the origin. It is easy to see that the 
embedding of $X$ uses at most $2|S_\varepsilon|$ slopes and $2|R_\varepsilon|$ 
triangles for faces (up to scaling). The embedding of the 
whole tripod will then require at most $6|S_\varepsilon|=\cO(\Delta^3)$ 
slopes and $6|R_\varepsilon|=\cO(\Delta^3)$ non-homothetic triangles. 

Let us estimate the number of possible points where a relevant vertex may be 
embedded. For every relevant double bubble, there are at most $D$ 
possibilities where its root edge may be embedded within the embedding of $X$. 
Since a bubble may be either above or below the spine, each relevant bubble 
has at most $2D$ possibilities where it may appear within $X$, and at 
most $6D$ possibilities within the whole tripod. As soon as we fix the 
embedding of the root edge and the relative position of the bubble with respect 
to its spine, we are left with at most $|P_\varepsilon|$ possibilities where a 
relevant vertex may be embedded. There are overall at most $6D 
|P_\varepsilon|= \cO(\Delta^2)$ possible embeddings of relevant vertices.

Using Lemma~\ref{lem-bubble}, it is straightforward to check that the 
embedding satisfies the required visibility properties. Lemma~\ref{lem-sp3t}
(and hence also Proposition~\ref{pro-sp3t} and 
Theorem~\ref{thm:tree}) is now proved.
\end{proof}

%
%
% ----------------------   Series-parallel of maximum degree three -----------------------------------
%
%

\section{Series-parallel graphs of maximum degree 3}
                                  
In this section, we prove Theorem~\ref{thm:spmax3}, which states that each 
series-parallel graph of maximum degree at most 3 has planar slope number 
at most 3. This bound is optimal, since it is not difficult to see that, e.g., 
the complete bipartite graph $K_{2,3}$, which is series-parallel, cannot be 
embedded with fewer than 3 slopes.

We will in fact show that any series-parallel graph with $\Delta\le 3$ can be 
embedded using the slopes from the set $S=\{0, \pi/4, -\pi/4\}$. This 
particular choice of $S$ is purely aesthetic, since for any other set $S'$ of 
three slopes there is an affine bijection of the plane that maps segments with 
slopes from $S'$ to segments with slopes from $S$. Thus, any plane graph that 
has an embedding with three distinct slopes also has an embedding with the 
slopes from the set~$S$.         

Throughout this section, segments of slope $\pi/4$ (or 0, or $-\pi/4$) will be 
known as \emph{increasing} (or \emph{horizontal}, or \emph{decreasing}, respectively).

Let us first define series-parallel graphs.
                        
A \emph{two-terminal graph} $(G,s,t)$ is a graph together with two distinct 
prescribed vertices $s,t\in V(G)$, known as \emph{terminals}. The vertex $s$ 
is called \emph{source} and $t$ is called \emph{sink}. 

For a sequence $(G_1,s_1,t_1), (G_2,s_2,t_2),\dotsc, (G_k,s_k,t_k)$ of 
two-terminal graphs, we define the \emph{serialization} of the sequence to be 
the two-terminal graph $(G,s_1,t_k)$ obtained by identifying, for every 
$i\in\{1,\dotsc,k-1\}$ the vertex $t_i$ with the vertex $s_{i+1}$. The 
\emph{parallelization} of the sequence of two-terminal graph is the 
two-terminal graph $(H,s,t)$ obtained by identifying all the sources $s_i$ 
into a single vertex $s$ and all the sources $t_i$ into a single vertex $t$. 
Whenever we perform parallelization of a sequence of graphs, we assume that at 
most one graph of the sequence contains the edge from source to sink. Thus, 
the result of a parallelization is again a simple graph. Serialization and 
parallelization will be jointly called \emph{SP-operations}.
                                                      
A two-terminal graph $(G,s,t)$ is called \emph{series-parallel graph} or 
\emph{SP-graph} for short, if it either consist of a single edge connecting 
the vertices $s$ and $t$, or if it can be obtained from smaller SP-graphs by 
an SP-operation. 

If follows from the definition, that SP-graphs can be constructed from single 
edges by repeated serializations and parallelizations. In general, this 
construction is not unique. E.g., a path of length four whose endpoints are 
the terminals can be constructed as a serialization of four edges, or as a 
serialization of two paths of length two. 

It is often convenient to employ special type of SP-operation that makes the 
construction of an SP-graph unique. To this end, we say that an SP-graph 
$(G,s,t)$ is obtained by a \emph{reduced serialization} if it is obtained as a 
serialization of a sequence of SP-graphs $(G_1,s_1,t_1),\dotsc,(G_k,s_k,t_k)$ 
where none of the operands $(G_i,s_i,t_i)$ can be expressed as a serialization 
of smaller graphs. Similarly, a \emph{reduced parallelization} is a 
parallelization whose operands are SP-graphs that cannot be expressed as 
parallelizations of smaller SP-graphs. It is not difficult to see that every 
SP-graph that is not a single edge can be uniquely expressed as a result of a 
reduced SP-operation.

Before proving the theorem, we give some useful definitions. For a 
pair of integers $j$ and $k$, we say that a series-parallel graph $(G,s,t)$ is 
a \emph{$(j,k)$-graph} if $G$ has maximum degree three, and furthermore, the 
vertex $s$ has degree at most $j$ and the vertex $t$ has degree at most~$k$.

Let us begin by a simple but useful lemma.
\begin{lemma}
\label{lem-11}Let $(G,s,t)$ be a $(1,1)$-graph. Then $G$ is either a single 
edge, a serialization of two edges, or a (not necessarily reduced) 
serialization of three graphs $G_1$, $G_2$ and $G_3$, where $G_1$ and $G_3$ 
consist of a single edge and $G_2$ is a $(2,2)$-graph.
\end{lemma}                                           
\begin{proof}
Assume $G$ is not a single edge. Then $G$ must have been obtained by a reduced 
serialization $H_1, H_2,\dotsc,H_k$ in which $H_1$ and $H_k$ consist of a 
single edge. If $k=2$, then $G$ is a serialization of two edges. If $k>2$, we 
let $G_1=H_1$, $G_2$ is the serialization of $H_2,\dotsc,H_{k-1}$, and 
$G_3=H_k$. The last case of the lemma then applies.
\end{proof}
               
We proceed with more terminology. An \emph{up-triangle} $abc$ is a right 
isosceles triangle whose hypotenuse $ab$ is horizontal and whose vertex $c$ is 
above the hypotenuse. We say that a series parallel graph $(G,s,t)$ 
has an \emph{up-triangle embedding} if it can be embedded inside an 
up-triangle $abc$ using the slopes from $S$, in such a way that the two vertices 
$s$ and $t$ coincide with the two endpoints of the hypotenuse of $abc$, and all 
the remaining vertices are either inside or on the boundary of~$abc$.
            
The concept of up-triangle embedding is motivated by the following lemma.

\begin{lemma}
\label{lem-up}
Every $(2,2)$-graph has an up-triangle embedding.
\end{lemma}                                      
\begin{proof}                        
Let $(G,s,t)$ be a $(2,2)$-graph. We proceed by induction on the size of $G$. If 
$G$ is a single edge, it obviously has an up-triangle embedding. Assume now that $G$ 
has been obtained by serialization of a sequence of graphs 
$G_1,G_2,\dotsc,G_k$. Since $G$ has maximum degree 3, all the graphs $G_i$ are 
necessarily $(2,2)$-graphs. By induction, all the graphs $G_i$ have an 
up-triangle embedding. We can join all these embeddings into a chain to obtain 
an up-triangle embedding of $G$ (see Figure~\ref{fig-up1}). 

\begin{figure}
\hfil\includegraphics[scale=0.6]{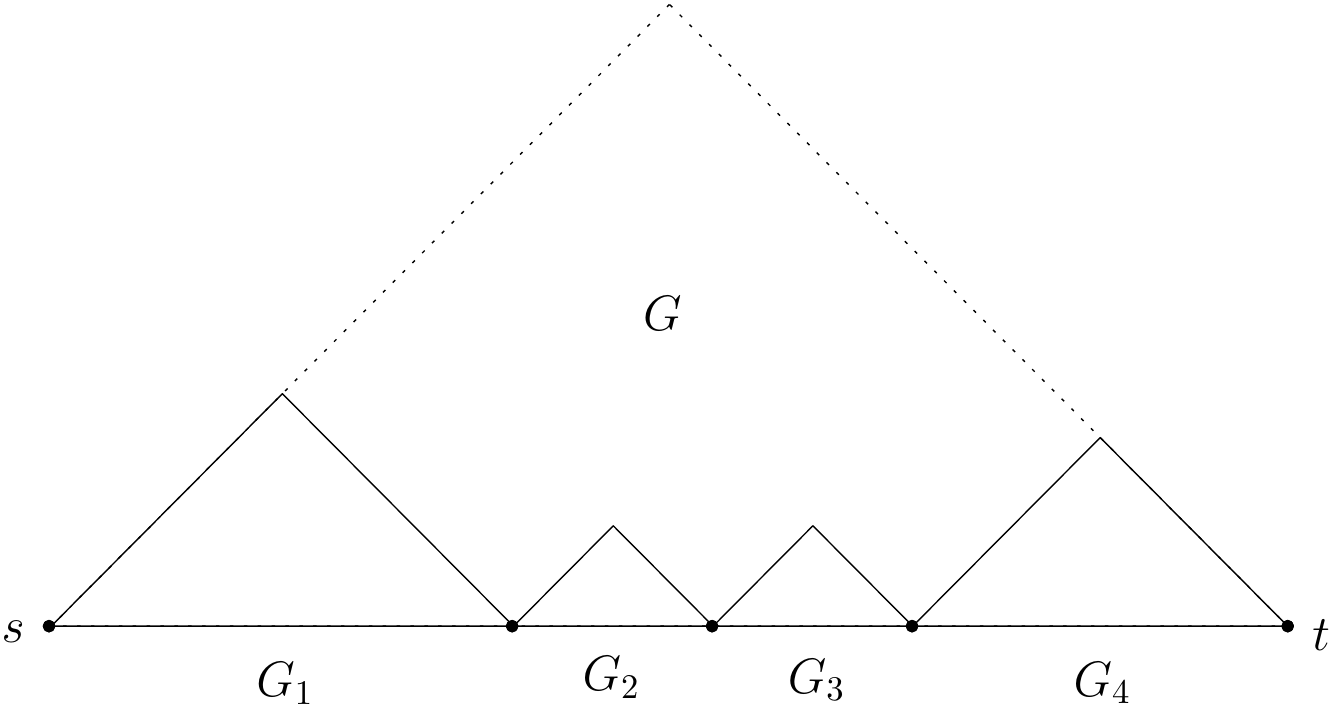}
\caption{Serialization of up-triangle embeddings yields an up-triangle 
embedding.}\label{fig-up1}
\end{figure}

Assume now that $G$ has been obtained by parallelization. Since $G$ is a 
$(2,2)$-graph, it must have been obtained by parallelizing two $(1,1)$-graphs 
$G_1$ and $G_2$. By Lemma~\ref{lem-11}, for each of the two graphs $G_i$ one of the following 
possibilities holds:
\begin{itemize}
\item $G_i$ is a single edge,
\item $G_i$ is a serialization of two edges $G_i^1$ and $G_i^2$, or
\item 
$G_i$ is a serialization of three graphs $G_i^1$, $G_i^2$ and 
$G_i^3$, where both $G_i^1$ and $G_i^3$ are single edges, and $G_i^2$ is a 
$(2,2)$-graph. By induction, we know that $G_i^2$ has an 
up-triangle embedding.
\end{itemize}    
In all the cases that may occur, we can obtain an up-triangle embedding of $G$ 
from the up-triangle embeddings of its subgraphs, as shown in 
Figure~\ref{fig-up2}.
\end{proof}             

\begin{figure}
\hfil\includegraphics[scale=0.6]{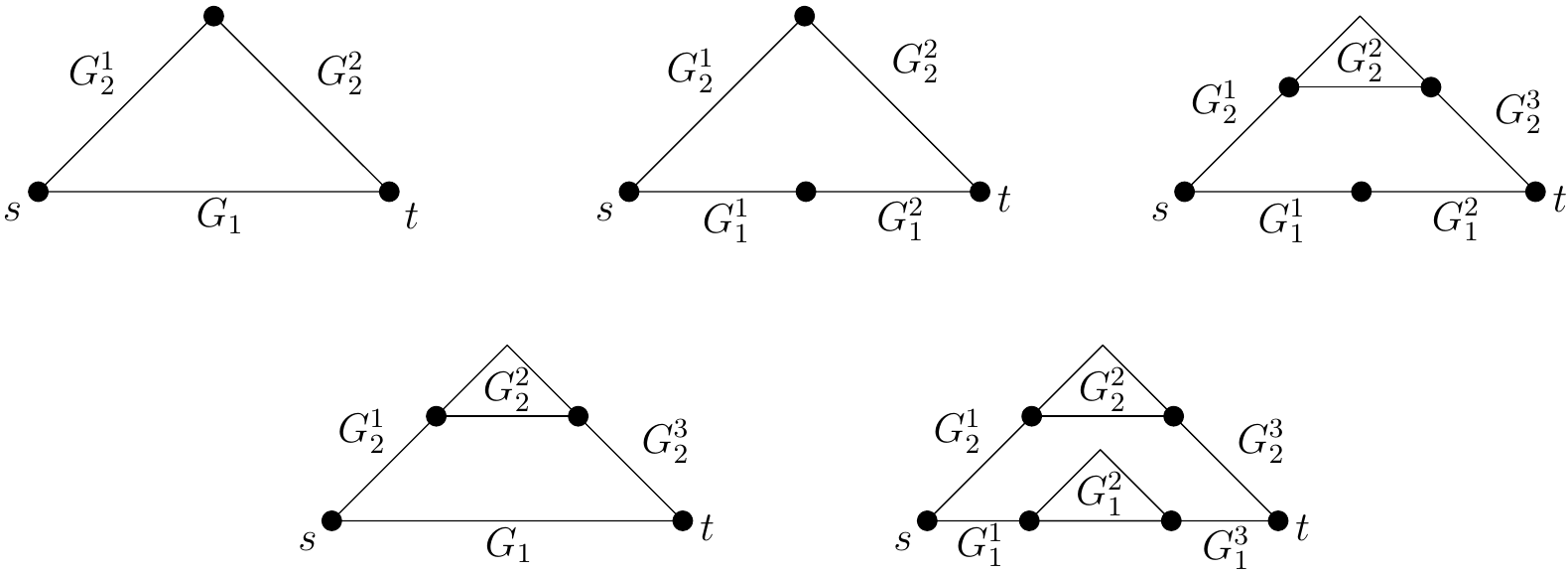} \caption{Possible construction of a 
$(2,2)$-graph $G$ by parallelization of two $(1,1)$-graphs $G_1$ and 
$G_2$.}\label{fig-up2}
\end{figure}

To deal with $(3,2)$-graphs, we need a more general concept than up-triangle 
embeddings. To this end, we introduce the following definitions. 

An \emph{up-spade} is a convex pentagon with vertices $a,b,c,d,e$ in 
counterclockwise order, such that the segment $ab$ is decreasing, the segment 
$bc$ is horizontal, the segment $cd$ is increasing, the segment $ed$ is 
decreasing and the segment $ae$ is increasing. We say that a series-parallel 
graph $(G,s,t)$ has an \emph{up-spade embedding} if it can be embedded into an 
up-spade $abcde$ using the slopes from $S$, in such a way that the vertex $s$ 
coincides with the point $a$, the vertex $t$ coincides either with the point 
$b$ or with the point $c$, and all the remaining vertices of $G$ are inside or 
on the boundary of the up-spade. Analogously, a \emph{reverse up-spade 
embedding} is an embedding of a series-parallel graph $(G,s,t)$ in which $s$ 
coincides with $b$ or $c$ and $t$ coincides with $d$. 
See Figure~\ref{fig-spade}.   

\begin{figure}
\hfil\includegraphics[scale=0.6]{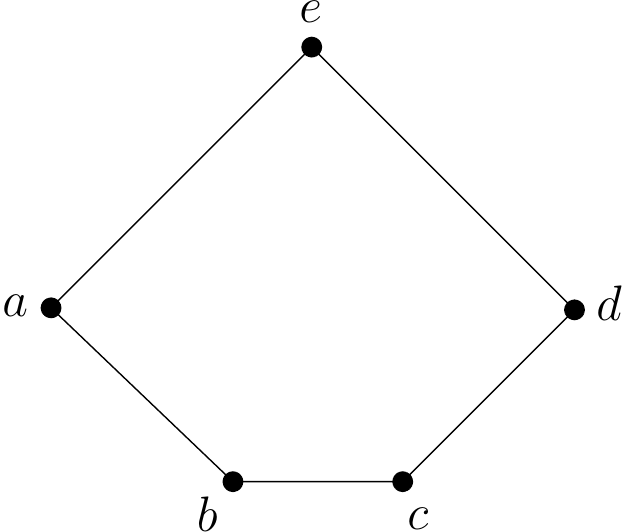} \caption{An up-spade.}\label{fig-spade}
\end{figure}
                               
\begin{lemma}
\label{lem-32} Every $(3,2)$-graph $(G,s,t)$ has an up-spade embedding or an 
up-triangle embedding. Similarly, every $(2,3)$-graph $(G,s,t)$ has a reverse 
up-spade embedding or an up-triangle embedding.
\end{lemma}
\begin{proof}                                 
It suffices to prove just the first part of the lemma; the other part is 
symmetric. We again proceed by induction.

Let $(G,s,t)$ be a $(3,2)$-graph. If $G$ is also a $(2,2)$-graph, then $G$ has 
an up-triangle embedding by Lemma~\ref{lem-up}. Assume that $G$ is not a 
$(2,2)$-graph. It is easy to see that in such case $G$ has no up-triangle 
embedding, since it is impossible to embed three edges into an up-triangle in 
such a way that they meet in the endpoint of its hypotenuse.

Assume that $G$ has been obtained by a reduced serialization of a sequence of 
graphs $G_1,G_2,\dotsc,G_k$. It follows that the graph $G_2$ is a single edge, 
because otherwise the two graphs $G_1$ and $G_2$ would share a vertex of 
degree at least 4. Let $G_3^+$ be the (possibly empty) serialization of 
$G_3,\dotsc,G_k$. If $G_3^+$ is nonempty, it has an up-triangle embedding by 
Lemma~\ref{lem-up}. The graph $G_1$ has an up-spade embedding by induction. We 
may combine these embeddings as shown in Figure~\ref{fig-321} to obtain an 
up-spade embedding of $G$. If $G_3^+$ is empty, the construction is even 
simpler.          

\begin{figure}
\hfil\includegraphics[scale=0.6]{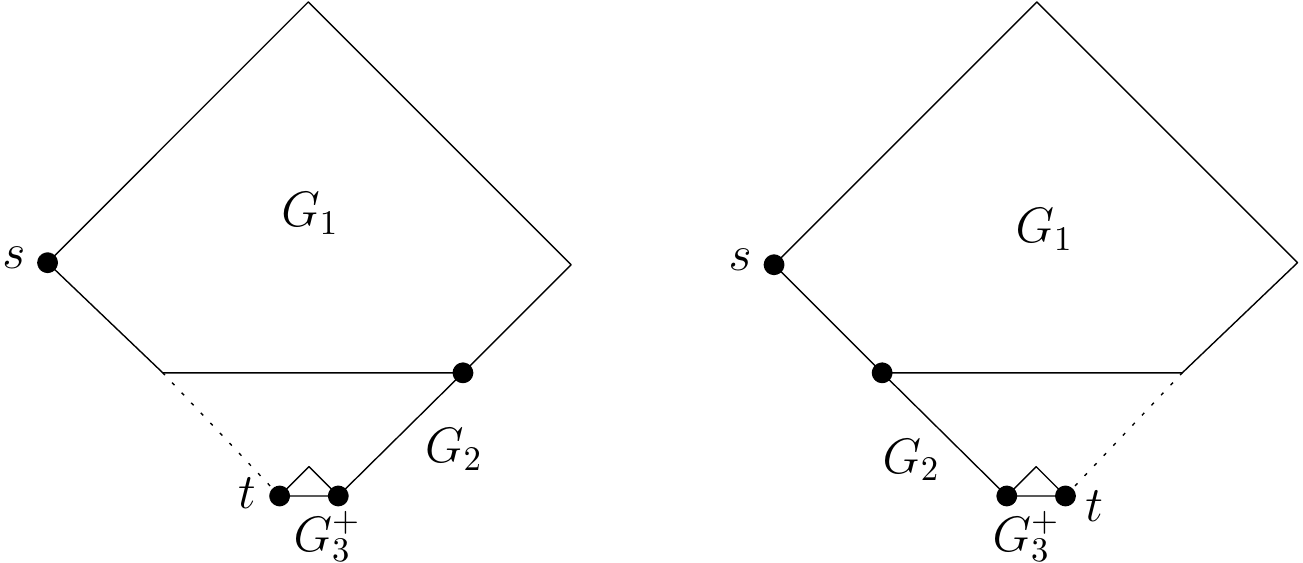} \caption{Constructing an up-spade 
embedding of a $(3,2)$-graph by serialization of a $(3,2)$-graph $G_1$, an edge 
$G_2$, and a $(2,2)$-graph $G_3^+$.}\label{fig-321}
\end{figure} 

Assume now that $G$ has been obtained by parallelization. Necessarily, it was 
a parallelization of a $(1,1)$-graph $G_1$ and a $(2,1)$-graph $G_2$. The 
graph $G_2$ can then be obtained by a (not necessarily reduced) serialization 
of a $(2,2)$-graph $G_2^1$ and a single edge $G_2^2$. The graph $G_2^1$ has an 
up-triangle embedding. Combining these embeddings, we obtain an up-spade 
embedding of $G$, as shown in Figure~\ref{fig-322}. Note that we distinguish 
the possible structure of $G_1$ using Lemma~\ref{lem-11}.
\end{proof}                                            

\begin{figure}
\hfil\includegraphics[scale=0.6]{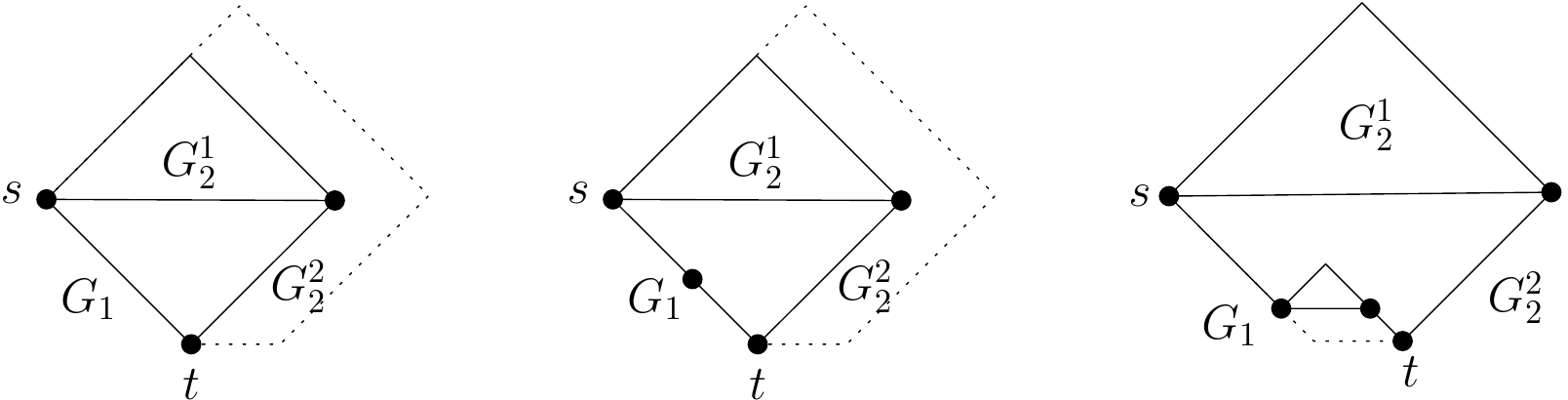} \caption{Constructing an up-spade 
embedding of a $(3,2)$-graph by parallelization of a $(1,1)$-graph $G_1$, and  
a $(2,1)$-graph $G_2$.}\label{fig-322}
\end{figure}

We are now ready to give the proof of the main theorem of this section.

\begin{proof}[Proof of Theorem~\ref{thm:spmax3}] Let $(G,s,t)$ be a
series-parallel graph of maximum degree at most 3. We may assume that both $s$ 
and $t$ have degree 3, otherwise we obtain the required embedding of $G$ directly 
from Lemma~\ref{lem-32}. 

Let us distinguish the possible constructions of $G$.

Assume first that $G$ was obtained by a parallelization of three graphs $G_1$, 
$G_2$ and $G_3$. Then all the three graphs $G_i$ are $(1,1)$-graphs.  By 
Lemma~\ref{lem-11}, each $G_i$ is a single edge, a series of two edges, or a 
series of an edge $G_i^1$, a $(2,2)$-graph $G_i^2$ and an edge $G_i^3$. Since 
at most one of the three graphs $G_1$, $G_2$ and $G_3$ consists of a single 
edge, we may easily construct an embedding of $G$. A typical case is shown in 
Figure~\ref{fig-331}, where $G_1$ is assumed to be a path of length 2, $G_2$ 
is a single edge, and $G_3$ is a serialization of an edge, a $(2,2)$-graph, 
and another edge. The remaining possibilities for the $G_i$'s are analogous.   

\begin{figure}
\hfil\includegraphics[scale=0.6]{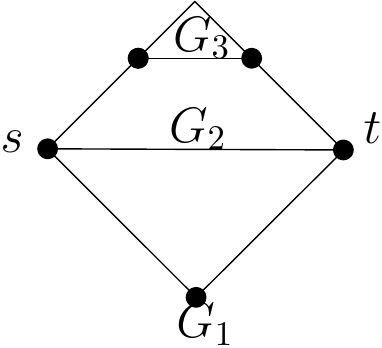} \caption{Example of a $(3,3)$-graph 
obtained by a parallelization of three graphs.}\label{fig-331}
\end{figure} 

Next, let us deal with the case when $G$ is a reduced parallelization of a 
$(1,1)$-graph $G_1$ and a $(2,2)$-graph $G_2$. Since the parallelization is 
reduced, we know that $G_2$ is obtained by serialization. Necessarily, at 
least one graph in the reduced serialization of $G_2$ consists of a single 
edge. From this, we conclude that $G_2$ can be obtained by a (not necessarily 
reduced) serialization of a $(2,2)$-graph $G_2^1$, an edge $G_2^2$, and a 
$(2,2)$-graph $G_2^3$. Using the fact that each $(2,2)$-graph has an 
up-triangle embedding, we construct the embedding of $G$ as shown in 
Figure~\ref{fig-332}.     

\begin{figure}
\hfil\includegraphics[scale=0.6]{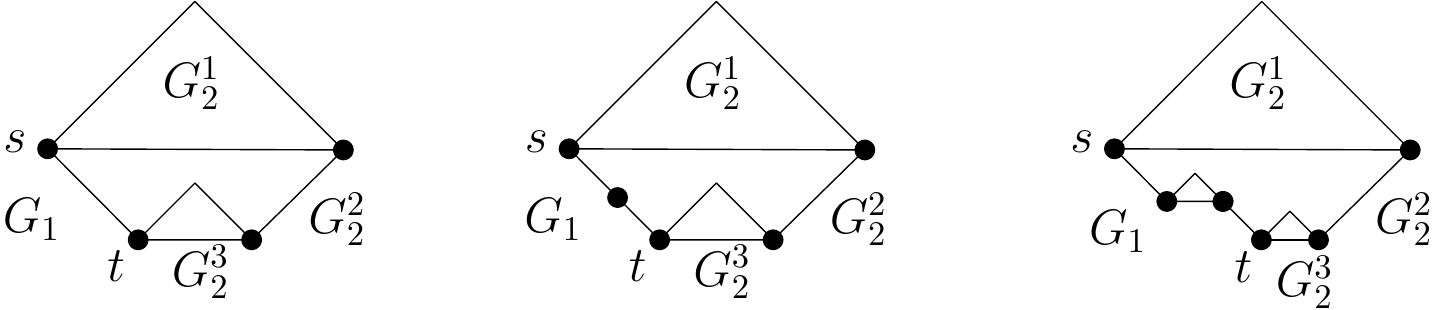} \caption{Embedding of a 
$(3,3)$-graph obtained by a reduced parallelization of a $(1,1)$-graph $G_1$ 
and a $(2,2)$-graph $G_2$. The three cases correspond to the three possible 
decompositions of $G_1$ by Lemma~\ref{lem-11}.}\label{fig-332}
\end{figure}

Now consider the situation when $G$ is obtained by a parallelization 
of a $(1,2)$-graph $G_1$ and a $(2,1)$-graph $G_2$. We see that $G_1$ must be 
a series of an edge $G_1^1$ and a $(2,2)$-graph $G_1^2$, while $G_2$ is a 
series of a $(2,2)$-graph $G_2^1$ and an edge $G_2^2$. We then obtain an 
embedding of $G$ by the construction depicted in Figure~\ref{fig-333}. 

\begin{figure}
\hfil\includegraphics[scale=0.6]{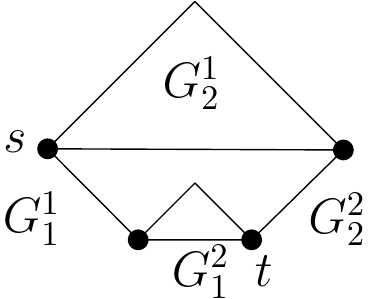} \caption{Embedding of a 
$(3,3)$-graph obtained by a reduced parallelization of a $(1,2)$-graph $G_1$ 
and a $(2,1)$-graph $G_2$.}\label{fig-333}
\end{figure}

It remains to consider the situation when $G$ is obtained by serialization. 
Necessarily, at least one graph in the reduced serialization of $G$ consists 
of a single edge. We then conclude that $G$ can be expressed as a serialization 
of a $(3,2)$-graph $G_1$, an edge $G_2$, and a $(2,3)$-graph $G_3$. We know by 
Lemma~\ref{lem-32} that $G_1$ has an up-spade embedding and that $G_3$ has a 
reverse up-spade embedding. Furthermore, we may flip the embedding 
of $G_3$ upside down, since this operation preserves the set of slopes $S$.
With the embedding of $G_1$ and the flipped embedding of $G_3$, we easily 
obtain an embedding of $G$, as shown in Figure~\ref{fig-334}. 

\begin{figure}
\hfil\includegraphics[scale=0.6]{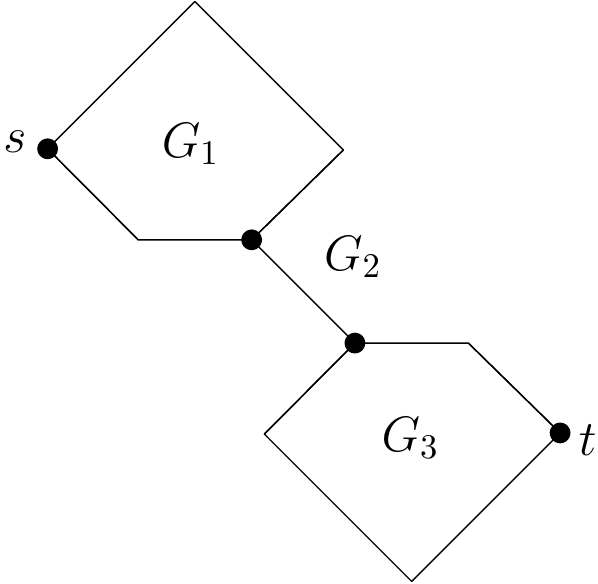} \caption{Embedding of a 
$(3,3)$-graph obtained by a serialization of a $(3,2)$-graph, an edge, and a 
$(2,3)$-graph.}\label{fig-334}
\end{figure}

This completes the proof of the main result of this section.
\end{proof} 

%
%
% ----------------------   The conclusion  -----------------------------------
%
%
\section{Conclusion and open problems}

We have presented an upper bound of $\cO(\Delta^5)$ for the planar slope number of planar partial
3-trees of maximum degree $\Delta$. It is not obvious to us if the used methods can be generalized to
a larger class of graphs, such as planar partial $k$-trees of bounded degree. Since a partial $k$-tree
is a graph of tree-width at most $k$, it would mean generalizing our result to graphs of a larger,
yet constant, tree-width.

In view of the results of Keszegh et al.~\cite{kppt2007} and Mukkamala and Szegedy~\cite{ms2009}
for the slope number of (sub)cubic planar graphs, it would also be interesting to find analogous
bounds for the planar slope number.

This paper does not address lower bounds for the planar slope number in terms
of $\Delta$; this might be another direction worth pursuing.

%
%
% ----------------------   The bibliography -----------------------------------
%
%

\end{document}